%% file: Hk-GenEO-paper.tex
\documentclass[11pt]{article}

\usepackage[a4paper]{geometry}
\usepackage[show]{ed}   

\usepackage{amsthm}   
\usepackage{xcolor}
\usepackage{url}
\usepackage{bm}

\newtheorem{theorem}{Theorem}[section]
\newtheorem{assumption}[theorem]{Assumption}
\newtheorem{corollary}[theorem]{Corollary}
\newtheorem{definition}[theorem]{Definition}
\newtheorem{notation}[theorem]{Notation}

\newtheorem{lemma}[theorem]{Lemma}

\newtheorem{proposition}[theorem]{Proposition}
\newtheorem{remark}[theorem]{Remark}

\usepackage{}
\usepackage{graphicx}
\usepackage{multirow,bigstrut}
\usepackage{amsmath,amsfonts,amssymb}
\usepackage{mathrsfs}
\usepackage{color}
\usepackage{upgreek}
\usepackage{bm}
\usepackage{verbatim}
\usepackage{mathtools}
\usepackage{calligra}
\usepackage{latexsym}
\usepackage{color}
\usepackage{soul}
\usepackage{authblk}
\usepackage{bm}
\usepackage{subcaption}

\usepackage{verbatim}
\usepackage{exscale}
\usepackage{relsize}

\usepackage{tikz}
\usetikzlibrary{shapes,arrows}

\usepackage{authblk}
\usepackage[colorlinks=true,allcolors=purple]{hyperref}

\newcommand{\assign}{\mathrel{\mathop:}=}

\numberwithin{equation}{section}
\newcommand\wt{\widetilde}
\newcommand\RR{\mathbb R}

\DeclareMathOperator{\diverg}{div}
\DeclareMathOperator{\supp}{supp}
\DeclareMathOperator{\Int}{Int}
\DeclareMathOperator{\linspan}{span}
\DeclareMathOperator{\dof}{dof}
\newcommand{\ovdof}{\overline{\dof}}
\newcommand{\rd}{\mathrm{d}}

\newcommand\CC{\mathbb C}

\newcommand{\Cstab}{C_{\textup{\textsf{stab}}}}
\newcommand{\amin}{a_{\textup{\textsf{min}}}}
\newcommand{\amax}{a_{\textup{\textsf{max}}}}

\renewcommand{\kappa}{k}

\begin{document}


\title{Schwarz preconditioner with $H_k$-GenEO coarse space for the indefinite Helmholtz problem}

\author{Victorita Dolean\protect\footnote{CASA, TU Eindhoven; 
    \texttt{v.dolean.maini@tue.nl}},
  Mark Fry\protect\footnote{Department of Mathematics and
    Statistics, University of Strathclyde, Glasgow, G1 1XH, UK;
    \texttt{mark.fry@strath.ac.uk, m.langer@strath.ac.uk}},
  Ivan G.~Graham\protect\footnote{Department of Mathematical Sciences, 
    University of Bath, Bath, BA2 7AY, UK; \texttt{i.g.graham@bath.ac.uk}},
  Matthias Langer$^\dag$
 }

\date{} 

\maketitle

\begin{abstract}
\noindent
GenEO (`Generalised Eigenvalue problems on the Overlap') is a method from the family of spectral coarse spaces that 
can efficiently rely on local eigensolves in order to build a robust parallel domain decomposition preconditioner for elliptic PDEs. 
When used as a preconditioner in a conjugate gradient, this method is extremely efficient in the positive-definite case, 
yielding an iteration count completely independent of the number of subdomains and heterogeneity.  
In a previous work this theory was extended to the cased of convection--diffusion--reaction problems, 
which may be non-self-adjoint and indefinite, and whose discretisations are solved with preconditioned GMRES. 
The GenEO coarse space was then defined here using a generalised eigenvalue problem based on a self-adjoint 
and positive definite subproblem.  The resulting method, called $\Delta$-GenEO becomes robust with respect to the 
variation of the coefficient of the diffusion term in the operator and depends only very mildly on variations 
of the other coefficients. However, the iteration number estimates get worse as the non-self-adjointness 
and indefiniteness of the operator increases, which is often the case for the high frequency Helmholtz problems. 
In this work, we will improve on this aspect by introducing a new version, called $H_k$-GenEO, which uses a 
generalised eigenvalue problem based directly on the indefinite operator which will lead to a robust method 
with respect to the increase in the wave-number.  We provide theoretical estimates showing the dependence of 
the size of the coarse space on the wave-number.  The new method is entirely robust with respect to $k$ and this is confirmed by numerical results.
\end{abstract}

\noindent
\textbf{MSC Classes:} \texttt{65N22, 65N55, 65F10}
\\[2ex]

\input{Introduction.tex}

\input{general_setting.tex}

\input{theoretical_tools.tex}

\input{main_result.tex}



\end{document}

%% file: Introduction.tex
\section{Introduction}
This paper is concerned with the theory and implementation of two-level domain decomposition preconditioners for finite element  
systems arising from the indefinite version of Helmholtz boundary value problems with highly variable coefficients defined 
on $\Omega \subset \mathbb{R}^d$, given by:

\begin{subequations}
\label{eq:problem}
	\begin{alignat}{2}
		-\diverg(A \nabla u) - k^2 u &= f \qquad && \text{in } \Omega, \\
		u &= 0 \qquad && \text{on } \partial\Omega.
	\end{alignat}
\end{subequations}
The computational domain is supposed to be polygonal or Lipschitz polyhedral domain $\Omega\subset \mathbb{R}^{d}$ ($d=2,3$). 
As well as a weak regularity assumption on the coefficient matrix $A$ which is positive definite, we will assume that
\eqref{eq:problem} has a unique weak solution $u\in H^1(\Omega)$ for all $f \in L^2(\Omega)$. 
A key parameter when solving this problem is the wavenumber $k$; the indefiniteness of the problem increases with $k$, 
becoming thus more challenging for linear solvers.

In overlapping domain  decomposition methods, the global domain $\Omega$ is covered by a set of
overlapping subdomains~$\Omega_i$, $i=1,\ldots,N$, and the classical one-level additive Schwarz preconditioner
is built from partial solutions on each subdomain.  Since this preconditioner is, in general, not scalable 
as the number of subdomains grows, an additional global coarse solve is usually added to enhance scalability, 
as well as robustness with respect to coefficient heterogeneity or the increase in the wavenumber.

\bigskip

\noindent 
\textbf{Aim of the paper and literature review.}
Coarse spaces are instrumental in achieving robustness and scalability in domain decomposition methods. 
It is now well known and commonly accepted that classical coarse spaces (piecewise polynomials on a coarse grid, 
e.g., \cite{Cai:1999:RAS,Toselli:2005:DDM}) are not robust for highly heterogeneous problems. 
For that reason, different authors introduced operator-dependent coarse spaces, which are better adapted to the heterogeneity;
see for example \cite{Graham:2007:DDM}.  In the past decade a new idea has emerged introducing the powerful concept 
of `spectral' coarse spaces, which are now built from well-chosen modes of local generalised eigenvalue problems defined at a subdomain level.
The use of spectral information for the solution of multiscale PDEs, as a tool of achieving robust approximation techniques, 
is not new \cite{Babuska:2011:OLA,Efendiev:2013:GMF}, but in domain decomposition preconditioning this was introduced in \cite{Galvisa:2010:DDP,Galvisb:2010:DDP}.  

This paper focuses on a variant of the GenEO coarse space, which was introduced for the first time in \cite{Spillane:2014:ARC} 
for general coercive, self-adjoint systems of PDEs.  In this seminal paper, it was rigorously proved that 
(when combined with a one-level method and using preconditioned conjugate gradients as a Krylov iterative solver) 
the resulting algorithm is not only scalable with respect to the number of subdomains but also robust with respect to coefficient variation.   
More recent works include for example \cite{Agullo:2019:RPG, Heinlein:2019:AGC,Bastian:2021:MSD} and \cite{Spillane:2021:TFA}, 
while more complete lists of contributions can be found in the literature surveys in, e.g., \cite{Spillane:2014:ARC} and \cite{Galvis:2018:ODD}.

While the theory for the above methods is very well established for SPD (symmetric positive definite) problems 
since the behaviour and the properties of local eigenfunctions are entirely known, for the indefinite or non-self-adjoint problems, 
the progress has not been slower but mostly numerical.  As an example for Helmholtz, the first publications on spectral coarse space 
include \cite{Conen:2014:ASH,Bootland:2019:OTD}, while a thorough comparison of the state-of-the-art coarse spaces can be found in \cite{Bootland:2021:ACS}. 
In this work, the resulting two level preconditioners are compared for large scale high-frequency Helmholtz problems. 
From this comparison it results that one of the most promising coarse spaces for Helmholtz problems
(also called `$H$-GenEO' in \cite{Bootland:2021:ACS,Bootland:2022:GCS}) has good robustness properties at high frequency, 
but, because of the non-self-adjoint nature of the operators involved, the theory was out of reach. 
Later on, the similarity between $H$-GenEO and $\Delta$-GenEO method observed numerically at lower frequencies in \cite{Bootland:2022:GCS} 
motivated the theoretical work in \cite{Bootland:2022:OSM} where the authors studied how well the GenEO technology works 
(both in theory and in practice) for general  non-self-adjoint and indefinite PDEs of the form \eqref{eq:delta_problem}:
\begin{subequations}
\label{eq:delta_problem}
	\begin{alignat}{2}
		-\diverg(A \nabla u) + \mathbf{b}\cdot \nabla u + cu &= f \qquad && \text{in } \Omega, \\
		u &= 0 \qquad && \text{on } \partial\Omega.
	\end{alignat}
\end{subequations}
The $\Delta$-GenEO method constructs a coarse space by solving a generalised eigenvalue problem using a \textit{nearby} SPD problem. 
Whilst the work cited has gone a long way to show that the GenEO coarse space provides an improvement for solving an SPD problem, 
and the $\Delta$-GenEO coarse space provides a similar improvement with a non-self-adjoint and indefinite problem, 
both of these scenarios are constructing the coarse space using an SPD operator.
The work being carried out in this paper is to develop the theory for the two-level domain decomposition preconditioner, 
using the $H_\kappa$-GenEO coarse space, which is adaptation of the $H$-GenEO coarse space first mentioned in \cite{Bootland:2021:ACS}. 
The $H_\kappa$-GenEO coarse space differs from previous methods as it uses the full problem operator to build the coarse space. 
Because of this, the usual estimates of the spectral properties are no longer applicable.  

After discretisation by finite elements, the linear systems arising from \eqref{eq:problem} are symmetric but indefinite. 
For this reason, we use GMRES as the iterative solver, the convergence analysis relies on the `Elman theory'
(\cite{Elman:1983:VIM} or \cite[Lemma C.11]{Toselli:2005:DDM}), which requires an upper bound for the norm of
the preconditioned matrix and a lower bound on the distance of its field of values from the origin. 
As a result, the number of GMRES iterations to achieve a given error tolerance is a function of these two bounds. 
Thus, the dependence of these bounds on the parameters of the problem and of the preconditioner 
(e.g., mesh size, subdomain size, overlap, heterogeneity, and wave-number) is of paramount interest 
when analysing the robustness and scalability of the preconditioner.
The use of the Elman theory in domain decomposition goes back to Cai and Widlund \cite{Cai:1992:DDA} (see also \cite[Section~11]{Toselli:2005:DDM}). \\

\noindent
\textbf{The main results of this paper.}  
In our work, we have used two different sizes of decomposition. One fine decomposition which is nested inside a larger coarse decomposition. Where in previous studies, the same size was used for both the fine and coarse decompositions. Our main theoretical result, Theorem~\ref{theorem: convergence}, provides rigorous and $k$-explicit upper bounds on the fine mesh
diameter $H^f$ and on the `eigenvalue tolerance'~$\tau$ (for the local GenEO eigenproblems) which ensure
that GMRES enjoys robust and mesh-independent convergence when applied to the preconditioned problem. 
We give here an example of the impact of these quantities on GMRES: if we reduce $H^f$, we increase the number of subdomains 
and these domains will become smaller. Not to mention that even if the solution exhibits 
higher frequencies in the global domain, on a fraction of it, the problem becomes simpler, low frequency or might even become positive definite. 
The other quantity, $\tau$ is an indicator of how many eigenfunctions are incorporated into the coarse space. 
Optimal choices of $\tau$ and $H^f$ are very often determined empirically but here we provide, for the first time, explicit limits depending on $k$. 

Since the result of the paper is quite technical, we summarise it here when applied to \eqref{eq:problem} and show the improvement 
with respect to the method introduced in \cite{Bootland:2022:OSM}. 
GenEO coarse spaces are usually based on the $m_i$ (dominant) eigenfunctions corresponding to the smallest
eigenvalues $\lambda^i_1 \le \lambda^i_2 \le \ldots \le \lambda_{m_i}^i$ of the generalised eigenvalue problem on the subdomain $\Omega_i$.  
To obtain a robust rate of convergence for GMRES that depends only on $\Lambda$ 
(the maximum number of times any point is overlapped by the subdomains $\Omega^f_{i,j}$), we need some conditions on $H^f$ and 
the eigenvalue tolerance $\tau \assign \min_{i=1}^N  \lambda^i_{m_i+1}$, where $N$ is the number of subdomains in the coarse mesh. 
In the equations below, $\Cstab>0$ is the stability constant for the problem \eqref{eq:problem} and is
defined in Assumption~\ref{Ass: 2_3}.  It blows up as the distance of $k^2$ from an eigenvalue of the Dirichlet problem for $-\diverg(A \nabla u)$ decreases.
The hidden constants usually depend only on $\Lambda$ and $H = H^c = H^f$ is the diameter of the mesh used.

\begin{itemize}

\item 
In \cite[Theorem~4.1]{Bootland:2022:OSM} it is proved that a robust rate is achieved if
\begin{equation}\label{eq:bdLam0}
    H \lesssim k^{-2} \qquad \text{and} \qquad  (1+\Cstab)^2 \,k^{8} \lesssim \tau.
\end{equation}
\item 
In \cite{Dolean:2024:ITE} the authors show that the bounds from \eqref{eq:bdLam0} can be improved:

\begin{equation}
    \label{eq:delta_results_improved}
	H \lesssim \kappa^{-1} \qquad \text{and} \qquad (1 + \Cstab)^2 \kappa^4 \lesssim \tau.
\end{equation}
\end{itemize}
In this paper, we show that the conditions for robustness are:
\begin{equation}
\label{Hk_results}
	H^f \lesssim \kappa^{-1} \qquad \text{and} \qquad (1 + \Cstab)^2 \kappa^2 \lesssim \tau.
\end{equation}
We thus show that, even though the theoretical bounds for the $\Delta$-GenEO given in \eqref{eq:delta_results_improved} 
are already an improvement with respect to the results in \cite{Bootland:2022:OSM}, $H_\kappa$-GenEO is still able to lead 
to better estimates than the $\Delta$-GenEO coarse space.  Using the $H_\kappa$-GenEO coarse space means that a 
significantly lower bound on $\tau$ can be used over the $\Delta$-GenEO. By using two separate decompositions, we have also been able to remove one of the constraints on the coarse space. The number of eigenvalues required for the coarse space no longer has any requirements on the size of the coarse mesh.

In all experiments from \cite{Bootland:2022:OSM} the condition on the eigenvalue tolerance in \eqref{eq:bdLam0} appears to be overly pessimistic, 
and the coarse space sizes necessary to achieve full robustness are rather moderate in practical computations.




%% file: general_setting.tex
\section{Useful Background}

\subsection{Problem formulation and discretisation}
The weak formulation of \eqref{eq:problem} is to find $u \in H^1_0(\Omega)$ such that 
\begin{equation}
\label{weak_form}
	b(u,v) = (f,v) \qquad \text{for all} \ v \in H^1_0(\Omega),
\end{equation}
where $f \in L^2(\Omega)$ and $b(\cdot,\cdot): H^1_0(\Omega) \times H^1_0(\Omega) \rightarrow \mathbb{R}$ is defined as 
\begin{equation*}
	b(u,v) = \int_\Omega (A\nabla u \cdot \nabla v - \kappa^2 uv)\,\rd x.
\end{equation*}
We shall be making use of the positive definite bilinear form $a(\cdot,\cdot): H^1_0(\Omega) \times H^1_0(\Omega) \rightarrow \mathbb{R}$,
\begin{equation*}
	a(u,v) = \int_\Omega A\nabla u \cdot \nabla v \,\rd x.
\end{equation*}
If $a$ and $b$ are defined on a subdomain $\Omega'$ of $\Omega$, we employ the notations $a_{\Omega'}$ and $b_{\Omega'}$.
The following weak regularity assumptions are used throughout this work.

\begin{assumption}\label{ass: 2_1}
	The coefficient $A$ and the right-hand side $f$ in problem \eqref{eq:problem} satisfy
	the following assumptions.
	\begin{enumerate}
		\item[\textup{(i)}] 
		$A: \Omega \rightarrow \mathbb{R}^{d \times d}$ is symmetric with $0 < a_{\textup{min}} \le a_{\textup{max}}$, s.t.\
		\begin{equation}
			\amin|\bm{\xi}|^2 \le A(x)\bm{\xi} \cdot {\bm \xi} \le \amax|\bm{\xi}|^2 
			\qquad \text{for all} \ x \in \Omega, \bm{\xi} \in \mathbb{R}^d.
		\end{equation}
		\item[\textup{(ii)}] 
		Without loss of generality, $\amin=1$ and the diameter, $D_\Omega$, of the domain $\Omega$ 
		satisfies $D_\Omega \le 1$.  \textup{(}If this is not the case, then the problem can be scaled accordingly.\textup{)}
	\end{enumerate}
\end{assumption}

\begin{notation}
For any subdomain $\Omega' \subset \Omega$ we are using $(\cdot,\cdot)_{\Omega'}$ to denote the $L^2(\Omega')$ inner product
and $\|\cdot\|_{\Omega}$ to denote the norm.  When $\Omega'= \Omega$, we write $(\cdot,\cdot)$ and $\|\cdot\|$ for the inner product 
and the norm respectively.  We also introduce the norm induced by the positive bilinear form $a$ and write $\|u\|_{a,\Omega'} \assign \sqrt{a_{\Omega'}(u,u)}$. 
When $\Omega'= \Omega$ we abandon the subscript $\Omega$.  We see that $b(u,v) = a(u,v) - \kappa^2 (u,v)$. 
We are also using the $\kappa$-inner product given by
\begin{equation*}
    (u,v)_{1, \kappa,\Omega'} \assign a_{\Omega'}(u,v) + \kappa^2 (u,v)_{\Omega'},
\end{equation*}
and we denote the induced $\kappa$-norm by $\|u\|_{1, \kappa, \Omega'}$.
\end{notation}

In all what follows solvability of \eqref{weak_form} is assumed.

\begin{assumption}
\label{Ass: 2_3}
We assume that, for any $f\in L^2(\Omega)$, the problem \eqref{weak_form} has a unique solution $u \in H_0^1(\Omega)$ and 
there exists a constant $\Cstab>0$ such that 
\begin{equation}
\label{eq: 2_10}
	\|u\|_{1,\kappa}  \le \Cstab \|f\| \qquad \text{for all} \ f\in L^2(\Omega). 
\end{equation}
\end{assumption}

The bilinear forms $a(\cdot,\cdot )$ and $(\cdot,\cdot)_{1,\kappa}$ are SPD; $b(\cdot,\cdot)$ is symmetric, but in general indefinite.

Let $\mathcal{T}_h$ be any shape regular triangular mesh over the domain $\Omega$.  For the purpose of this work, 2 or 3-dimensional simplices 
are considered, but this could easily be applied to $d$-dimensional simplices, where the maximum diameter is $h$. 
Let $V^h \subset H^1_0(\Omega)$ be any conforming finite element space.
The Galerkin approximation of \eqref{weak_form} is to find $u_h \in V^h$ such that 
\begin{equation}
\label{eq: 2_11}
	b(u_h,v) = (f,v) \qquad \text{for all} \ v \in V^h.
\end{equation}
If $n$ denotes the dimension of $V^h$, with a basis given by $\lbrace \phi_i \rbrace_{i=1}^n$, then (\ref{eq: 2_11}) can be represented by the linear system
\begin{equation}
\label{eq: 2_12}
	\mathbf{B}\mathbf{u} = \mathbf{f};
\end{equation}
the matrix $\mathbf{B}$ and the vector $\mathbf{f}$ are defined in terms of the basis functions: 
$(\mathbf{B})_{ij} \assign b(\phi_j, \phi_i)$ and $(\mathbf{f})_i \assign (f, \phi_i)$. 
Later we also need the matrices $\mathbf{A}$ and $\mathbf{S}$ that correspond to $a$ and $(\cdot,\cdot)$, 
namely, $(\mathbf{A})_{ij} \assign a(\phi_j, \phi_i)$ and $(\mathbf{S})_{ij} \assign (\phi_j, \phi_i)$ respectively. 

The solvability of \eqref{eq: 2_11} is assured by the following lemma from \cite[Theorem~2]{Schatz:1996:SNE}. 
This is required due to the indefiniteness of \eqref{eq: 2_11}.

\begin{lemma}[Schatz and Wang, 1996]
\label{schatz_wang}
Let Assumptions \ref{ass: 2_1} and \ref{Ass: 2_3} hold.  Then there exists an $h_0>0$ such that, for each $h$ with $0 < h < h_0$,
the problem \eqref{eq: 2_11} has a unique solution $u_h \in V_H$.  
Moreover, let $u$ be the unique solution of \eqref{weak_form}.  Then, for every $\varepsilon>0$ there exists $h_1=h_1(\varepsilon)>0$
such that, for every $h\in(0,h_1)$,
\begin{equation}
\label{eq: 3_13_a}
	\|u - u_h\| \le \varepsilon \|u - u_h\|_{{H^1}(\Omega)}
\end{equation}
and
\begin{equation}
\label{eq: 2_13}
	\|u - u_h\|_{H^1(\Omega)} \le \varepsilon \|f\|.
\end{equation}
\end{lemma}

\medskip

\noindent
Let us also recall the Friedrichs inequality \cite[Theorem 13.19]{Leoni:2017:FCS}, which will be used extensively throughout this work.

\begin{lemma}[Friedrichs inequality]\label{lemma: 2.5}
Let $\Omega' \subset \mathbb{R}^d$ be an open set that lies between two parallel hyperplanes with and 
let $L$ be the distance between the two hyperplanes.  Then, for all $u \in H^1_0(\Omega')$, 
\begin{equation}
\label{eq: 2_15}
	\|u \|_{\Omega'} \le \frac{L}{\sqrt{2}\,} \|\nabla u\|_{\Omega'}. 
\end{equation}
\end{lemma}

Combining \eqref{eq: 2_15} with Assumption~\ref{ass: 2_1} we obtain that, for any subdomain $\Omega' \subset \Omega$ with diameter $H$, 
the following estimate is true:
\begin{equation} \label{eq: 2_16}
	\|u\|_{\Omega'} \le \frac{H}{\sqrt{2}\,} \|\nabla u\|_{\Omega'} \le \frac{H}{\sqrt{2}\,} \|u\|_{1,k,\Omega'} 
	\qquad \text{for all} \ u \in H^1_0(\Omega').
\end{equation}

\subsection{Abstract spectral theory for the generalised indefinite eigenvalue problems}

Due to the nature of the problem being analysed, and the type of spectral coarse space we are going to develop, we will be dealing with indefinite eigenvalue problems. The following lemma demonstrates certain properties that we will be able to exploit.

\begin{lemma}\label{lem_ind_ev}
Let $B$ and $C$ be self-adjoint operators acting on a finite dimensional Hilbert space $V$, with symmetric forms defined as
\[
	\texttt{b}(v,w) := (Bv,w), \qquad c(v,w) := (Cv,w),
\]
and
\begin{equation}\label{E601}
	C \ge 0, \qquad \ker B\cap\ker C=\{0\}.
\end{equation}
Consider the generalised eigenvalue problem
\begin{equation}\label{ev_equ500}
	Bv = \lambda Cv.
\end{equation}
The generalised eigenvalue problem \eqref{ev_equ500} is non defective, i.e. it
has a full set of eigenvectors, and any $v \in V$ can be written as 
\begin{equation}\label{E606}
	v = v_0 + \sum_{i=1}^r \alpha_i p_i.
\end{equation}

\end{lemma}

\begin{proof}
    It follows from \eqref{E601} that there exists $\gamma\ge0$ such that
    \[
        \wt B \assign B+\gamma C
    \]
    is invertible.  For $v\ne0$ we have the following equivalences
    \begin{equation}\label{E602}
        Bv = \lambda Cv \quad\Leftrightarrow\quad (B+\gamma C)v = (\lambda+\gamma)Cv
        \quad\Leftrightarrow\quad \wt B^{-1}Cv = \frac{1}{\lambda+\gamma}c;
    \end{equation}
    note that $\lambda+\gamma\ne0$ if the second equality holds.
    Set
    \[
        U \assign \wt B^{-1}C, \qquad \mu \assign \frac{1}{\lambda+\gamma}.
    \]
    Then the relations in \eqref{E602} are equivalent to $Uv=\mu v$.
    
    Now set
    \[
        \wt b(v,w) \assign (\wt Bv,w), \qquad c(v,w) \assign (Cv,w).
    \]
    Then
    \begin{equation}\label{E603}
        \wt b(Uv,w) = (\wt B\wt B^{-1}Cv,w) = (Cv,w) = c(v,w)
    \end{equation}
    and hence
    \[
        \wt b(Uv,w) = (Cv,w) = (v,Cw) = (v,\wt B\wt B^{-1}Cw) = (\wt Bv,Uw) = \wt b(v,Uw),
    \]
    i.e.\ $U$ is symmetric with respect to the form $\wt b$.
    Since the form $c$ is positive semi-definite, we can apply the Cauchy--Schwarz inequality to obtain
    \begin{equation}\label{E604}
        |c(v,w)|^2 \le c(v,v)c(w,w).
    \end{equation}
    
    Now assume that $Uv=\mu v$ with $\mu\ne0$ and $v\ne0$.  It follows from \eqref{E603} and \eqref{E604} that, 
    for arbitrary $w$,
    \begin{align*}
        |\mu\wt b(v,w)|^2 &= |\wt b(\mu v,w)|^2 = |\wt b(Uv,w)|^2 = |c(v,w)|^2
        \le c(v,v)c(w,w)
        \\[1ex]
        &= \wt b(Uv,v)c(w,w) = \mu\wt b(v,v)c(w,w).
    \end{align*}
    Since $\wt B$ is invertible, we can find $w$ such that $\wt b(v,w)\ne0$, which implies that
    \[
        \mu\wt b(v,v)c(w,w) > 0.
    \]
    Since $\wt b(v,v)c(w,w)\in\RR$, this yields 
    \begin{equation}\label{E605}
        \mu \in \RR, \qquad \wt b(v,v) \ne 0,
    \end{equation}
    and consequently,
    \[
        c(v,v) = \wt b(Uv,v) = \mu\wt b(v,v) > 0.
    \]
    In particular, all eigenvalues of $U$ are real.
    
    Next we show that the non-zero eigenvalues are semi-simple.  Assume that we have a Jordan chain; 
    then there exist $v\ne0$, $\mu\ne0$ and $\hat v$ such that
    \[
        Uv = \mu v, \qquad U\hat v-\mu\hat v = v.
    \]
    It follows that
    \[
        \wt b(v,v) = \wt b(U\hat v-\mu\hat v,v) = \wt b(U\hat v,v)-\mu\wt b(\hat v,v)
        = \wt b(\hat v,Uv)-\mu\wt b(\hat v,v) = 0,
    \]
    which contradicts \eqref{E605}.  Hence $\mu$ is a semi-simple eigenvalue.
    
    Let $\mu_1,\mu_2$ be eigenvalues of $U$ with $\mu_1\ne\mu_2$ and eigenvectors $v_1,v_2$.  Then
    \[
        \mu_1\wt b(v_1,v_2) = \wt b(\mu_1v_1,v_2) = \wt b(Uv_1,v_2) = \wt b(v_1,Uv_2)
        = \wt b(v_1,\mu_2v_2) = \mu_2\wt b(v_1,v_2),
    \]
    which implies that $\wt b(v_1,v_2)=0$.  From this we obtain
    \[
        c(v_1,v_2) = \wt b(Uv_1,v_2) = \mu_1\wt b(v_1,v_2) = 0
    \]
    and hence $(Bv_1,v_2)=0$.
    
    Now let $v_0,v_1,\ldots,v_l$ be a Jordan chain corresponding to the eigenvalue $0$, i.e.\
    \[
        Uv_0 = 0, \quad Uv_1 = v_0, \;\ldots,\; Uv_l = v_{l-1},
    \]
    and let $w$ be an eigenvector corresponding to the eigenvalue $\mu\ne0$.  Then
    \[
        \mu^j\wt b(v_j,w) = \wt b(v_j,\mu^j w) = \wt b(v_j,U^jw) = \wt b(U^jv_j,w) = \wt b(v_0,w) = 0,
    \]
    and hence also $c(v_j,w)=0$.

    Let $X_1$ be the linear span of all eigenvectors corresponding to non-zero eigenvalues and 
    let $X_0$ be the algebraic eigenspace corresponding to the eigenvalue $0$.
    Further, let us normalise the eigenvectors corresponding to non-zero eigenvalues such that
    \[
        c(p_i,p_j) = \delta_{ij}.
    \]
    
    Every $v\in V$ can be written as
    \begin{equation}
        v = v_0 + \sum_{i=1}^r \alpha_i p_i
    \end{equation}
    with $v_0\in X_0$ and $\alpha_i\in\CC$.
\end{proof}

We can now define projection operators that satisfy suitable stability estimates. These projection operators are crucial in the analysis of the proposed preconditioner. 

\begin{lemma}\label{lem_abs_proj}
    Let $\texttt{b}$ and $c$ be as defined in Lemma \ref{lem_ind_ev}, where the eigenvalues, $(\lambda_l)_{l=1}^n$, and corresponding eigenvectors, $(p_l)_{l=1}^n$ are from the generalised eigenvalue problem \eqref{ev_equ500}. Then, for any $m\in\{r+1,\ldots,n-s\}$ such that $\lambda_{m+1}>0$, the local projector, $\Pi$, defined as
    \begin{equation}\label{eq_abs_proj}
        \Pi v = v_0 + \sum_{l=1}^m \alpha_l p_l, \qquad v\in V,
    \end{equation}
    satisfies
    \begin{equation}
        c(v-\Pi v,v-\Pi v) \le \frac{1}{\lambda_{m+1}} \texttt{b}(v-\Pi v,v-\Pi v).
    \end{equation}
\end{lemma}

\begin{proof}
    With $v$ as in \eqref{E606}, then
    \begin{align*}
	\textit{\texttt{b}}(v-\Pi v,v-\Pi v) &=\textit{\texttt{b}}\Biggl(\sum_{i=m+1}^r \alpha_l p_l,\sum_{i=m+1}^r \alpha_l p_l\Biggr)
	= \sum_{i=m+1}^r |\alpha_l|^2 \textit{\texttt{b}}(p_l,p_l)
	\\[1ex]
	&= \sum_{i=m+1}^r |\alpha_l|^2(Bp_l,p_l) = \sum_{i=m+1}^r |\alpha_l|^2(\lambda_lCp_l,p_l)
	= \sum_{i=m+1}^r \lambda_l|\alpha_l|^2c(p_l,p_l)
	\\[1ex]
	&\ge \lambda_{m+1}\sum_{i=m+1}^r |\alpha_l|^2c(p_l,p_l)
	= \lambda_{m+1}c\Biggl(\sum_{i=m+1}^r \alpha_l p_l,\sum_{i=m+1}^r \alpha_l p_l\Biggr)
	\\[1ex]
	&= \lambda_{m+1}c(v-\Pi v,v-\Pi v).
\end{align*}
\end{proof}

\subsection{Domain Decomposition}
In order to construct the two-level Schwarz preconditioner, we first need to partition the global domain, $\Omega$, into a set of first-level and coarse-level (second-level) subdomains. We also need a method of how to expand the partitioned subdomains so that they overlap. 
\begin{definition} \label{def: 2_3}
	Given a subdomain $D' \subset \Omega$, which is resolved by the chosen mesh, the extension of $D'$ by a layer of elements is 
	\begin{equation*}
		D = \Int \!\left( \bigcup_{\lbrace \ell \mid \supp(\phi_\ell) \cap D' \not= \emptyset \rbrace} \supp(\phi_\ell) \right)
	\end{equation*}
	where $\Int(\cdot)$ is the interior of domain.  Extension by multiple layers can then be achieved by applying this recursively. 
\end{definition}

It is being assumed that the global mesh, $\mathcal{T}_h$, is sufficiently fine to resolve the subdomains. To construct the coarse-level decomposition, the global domain is partitioned into a set of $N$ subdomains, $\lbrace {\Omega_i^c}' \rbrace_{i =1}^N$. These domains are expanded by one or more layers of mesh elements, in the sense of Definition \ref{def: 2_3}, creating the overlapping set of subdomains $\lbrace {\Omega_i^c} \rbrace_{i =1}^N$. The first-level decomposition is constructed by taking a subdomain from the coarse decomposition, ${\Omega_i^c}$, then partitioning this into a set of $Q_i$ subdomains, which are expanded again using Definition \ref{def: 2_3} such that
\begin{equation*}
    {\Omega_i^c} = \bigcup_{j=1}^{Q_i} \Omega_{i,j}^f,
\end{equation*}
for each $i \in 1, \ldots , N$. For the domains ${\Omega_i^c}$ and ${\Omega_{i,j}^f}$, it is possible to define the spaces,
\begin{equation}
\label{eq: 2_17}
	\widetilde{V}_i^c \assign \bigl\{v|_{\Omega_i^c} : v \in V^h\bigr\} \subset H^1(\Omega_i^c) 
	\qquad \text{and} \qquad 
	V_i^c \assign \bigl\{v \in \widetilde{V}_i^c : v|_{\partial \Omega_i^c} = 0\bigr\} \subset H^1_0(\Omega_i^c). 
\end{equation}
\begin{equation}
\label{eq: 2_17_f}
	\widetilde{V}_{i,j}^f \assign \bigl\{v|_{\Omega_{i,j}^f} : v \in V^h\bigr\} \subset H^1(\Omega_{i,j}^f) 
	\qquad \text{and} \qquad 
	V_{i,j}^f \assign \bigl\{v \in \widetilde{V}_{i,j}^f : v|_{\partial \Omega_{i,j}^f} = 0\bigr\} \subset H^1_0(\Omega_{i,j}^f). 
\end{equation}

\begin{notation}
    To simplify notation, we will let D be either $\Omega_i^c$ or $\Omega_{i,j}^f$, $\widetilde{V}_D$ be either $\widetilde{V}_i^c$ or $\widetilde{V}_{i,j}^f$, and $V_D$ be either $V_i^c$ or $V_{i,j}^f$.
\end{notation}

For each subdomain, ${\Omega_i^c}$ and ${\Omega_{i,j}^f}$, we denote its diameter by $H_i^c$ and $H_{i,j}^f$, and we set $H^c \assign \max_{i=1}^N H_i^c$ and $H^f \assign \max_{i=1}^N \max_{j=1}^{Q_i} H^f_{i,j}$. 
For any $u,v \in \widetilde{V}_D$, the local bilinear forms can be defined
\begin{equation}
	a_{D}(u,v) \assign \int_{D} A\nabla u \cdot \nabla v \,\rd x, \qquad
	b_{D}(u,v) \assign \int_{D} \bigl(A\nabla u \cdot \nabla v - \kappa^2 uv\bigr)\,\rd x, 
    \label{eq:2.18}
\end{equation}
and, for $k>0$, the local norm $\|\cdot\|_{1,k,D}$ is induced by the scalar product
\[
	(u,v)_{1, \kappa, D} \assign a_{D}(u,v) + k^2(u,v)_{D}.
\]

\begin{remark}
Although $b_{D}(\cdot , \cdot )$ is generally indefinite, for small enough diameter of D this will become positive definite. 
This is discussed in Lemma~\ref{lemma_3_4}.
\end{remark}

For any $v \in V_D$, let $E_D v$ denote its zero extension. For both the fine-level and coarse-level, this will be the extension to the whole of the domain, $\Omega$. Then,
\begin{equation}\label{eq: 2_20}
\begin{array}{ll}
    E_i^c: V_i^c \rightarrow V^h, &  i = 1, \ldots ,N, \\
    E^f_{i,j}: V^f_{i,j} \rightarrow V^h, &  j = 1, \ldots ,Q_i,
\end{array}
\end{equation}
The $L^2(\Omega)$ adjoint of the extension operator is called restriction operator, 
\begin{equation*}
\begin{array}{ll}
    R_i^c: V^h \rightarrow V_i^c, &  i = 1, \ldots ,N, \\
    R^f_{i,j}:  V^h \rightarrow V^f_{i,j}, &  j = 1, \ldots ,Q_i,
\end{array}
\end{equation*}
By use of the extension operator, the restriction of the bilinear forms to $V_D$ can be written as 
\[
	a_{\Omega_D}(u, v) = a(E_D u, E_D v), \qquad 
	b_{\Omega_D}(u, v) = b(E_D u, E_D v), \qquad 
	(u, v)_{\Omega_D} = (E_D u, E_D v)
\]
for all $u,v\in V_D$.
The one-level additive Schwarz preconditioner can now be given in matrix form as 
\begin{equation} \label{eq: 2_22}
	\mathbf{M}^{-1}_{AS,1} = \sum^N_{i=1} \sum_{j=1}^{Q_i} \mathbf{E}^f_{i,j} (\mathbf{B}^f_{i,j})^{-1} \mathbf{R}^f_{i,j}, \quad \text{where } \mathbf{B}^f_{i,j} = \mathbf{R}^f_{i,j} \mathbf{B} \mathbf{E}^f_{i,j}.
\end{equation}
Here, $\mathbf{E}_{i,j}^f$ and $\mathbf{R}_{i,j}^f$ denote the matrix representations of $E_{i,j}^f$ and $R_{i,j}^f$ with respect to the 
basis functions $\lbrace \phi_i \rbrace_{i=1}^n$ of $V^h$ and some basis functions of $V_{i,j}^f$.

In order to improve the effectiveness of the preconditioner, a coarse space is added.  
This improves the global exchange of information between the subdomains.  
Let $V_0 \subset V^h$ be such a coarse space, let $E_0: V_0 \rightarrow V^h$ be the natural embedding, and let $R_0$ be the $L^2$ adjoint of $E_0$,
\begin{equation*}
	(v_0, R_0 w) = (E_0 v_0,w) \qquad \text{for all} \ w \in V^h, v_0 \in V_0.
\end{equation*}
The two-level additive Schwarz preconditioner can now be given in matrix form as, 
\begin{equation} \label{eq: 2_23}
	\mathbf{M}^{-1}_{AS,2} = \mathbf{E}_0 \mathbf{B}^{-1}_0 \mathbf{R}_0 + \mathbf{M}^{-1}_{AS,1} , \quad \text{where } \mathbf{B}_0 = \mathbf{R}_0 \mathbf{B} \mathbf{E}_0.
\end{equation}
The preconditioned linear system from \eqref{eq: 2_12} reads as,
\begin{equation} \label{eq: 2_24}
	\mathbf{M}^{-1}_{AS,2} \mathbf{B} = \mathbf{M}_{AS,2}^{-1} \mathbf{f}.
\end{equation}

It is now possible to define the projectors used in the analysis. 
For each $i = 1, \ldots, N$ and $j = 1, \ldots, Q$, the projectors, $T^f_{i,j}: V^h \rightarrow V^f_{i,j}$, are defined by, 
\begin{equation}
	b_{\Omega^f_{i,j}}(T^f_{i,j} u, v) = b_{\Omega_i^c}(u, E^f_{i,j} v), \quad \forall v \in V^f_{i,j}. \label{eq: bT} 
\end{equation}
The projector $T_0: V^h \rightarrow V_0$, is defined as
\begin{equation}
	b_{\Omega_{0}}(T_0 u, v) = b(u, E_0 v), \quad \forall v \in V_0. \label{eq: 2_25_zero}
\end{equation}
where $u\in V^h$ and $\Omega_0=\Omega$.  Sufficient conditions for the existence of $T^f_{i,j}$
are given in Lemmas~\ref{lemma_3_4} and \ref{lemma_3_5}.

Given the operators $T^f_{i,j}$ and $T_0$, the operator $T: V^h \rightarrow V^h$ is defined as 
\begin{equation}
	T = E_0T_0 + \sum_{i=1}^N \sum_{j=1}^{Q_i} E^f_{i,j} T^f_{i,j}. \label{eq: 2_26}
\end{equation}
This allows the two-level additive Schwarz preconditioner to be represented in terms of the projector operator, $T$, as follows.

\begin{proposition}\label{prop: 2_6}
For any $u,v \in V^h$, with corresponding nodal vectors $\mathbf{u},\mathbf{v} \in \mathbb{R}^n$, 
\begin{equation} \label{eq: 2_27}
	\langle \mathbf{M}_{AS,2}^{-1} \mathbf{B} \mathbf{u}, \mathbf{v} \rangle_{\mathbf{D}_\kappa} = (T u, v)_{1, \kappa},
\end{equation}
where $\langle \cdot , \cdot \rangle_{\mathbf{D}_\kappa}$ is the inner product on $\mathbb{R}^n$ and the matrix $\mathbf{D}_\kappa$ given by
\begin{equation}\label{def:Dk}
    \mathbf{D}_\kappa \assign \mathbf{A} + \kappa^2  \mathbf{S}.
\end{equation}
\end{proposition}

\begin{proof}
Let $\mathbf{T}^f_{i,j}$ be the representation of $T^f_{i,j}$ with respect to the basis functions $\lbrace \phi_i \rbrace_{i=1}^n$ of $V^h$. 
In matrix form equation \eqref{eq: bT}
becomes
\begin{equation*}
	(\mathbf{v}^f_{i,j})^T \mathbf{B}^f_{i,j} \mathbf{T}^f_{i,j} \mathbf{u} = (\mathbf{v}^f_{i,j})^T \mathbf{R}^f_{i,j} \mathbf{B} \mathbf{u} \quad \forall \mathbf{u},\mathbf{v} \in \mathbb{R}^d .
\end{equation*}
%
meaning that $\mathbf{T}^f_{i,j}$ can be written as $\mathbf{T}^f_{i,j} = (\mathbf{B}^f_{i,j})^{-1} \mathbf{R}^f_{i,j} \mathbf{B}$. A similar result can be found for the $T_0$ operator, $\mathbf{T}_0 = (\mathbf{B}_0)^{-1} \mathbf{R}_0 \mathbf{B}$. By using the definition of $\mathbf{T}$ this gives,
\begin{equation*}
	\mathbf{T} =\mathbf{E}_0 \mathbf{T}_0 + \sum_{i=1}^N \sum_{j=1}^{Q_i} \mathbf{E}^f_{i,j} \mathbf{T}^f_{i,j} = \mathbf{E}_0 (\mathbf{B}_0)^{-1} \mathbf{R}_0 \mathbf{B} +  \sum_{i=1}^N \sum_{j=1}^{Q_i} \mathbf{E}^f_{i,j} (\mathbf{B}^f_{i,j})^{-1} \mathbf{R}^f_{i,j} \mathbf{B} = \mathbf{M}^{-1}_{AS,2} \mathbf{B}.
\end{equation*}
Therefore
\begin{align*}
	(Tu,v)_{1,\kappa} &= a(Tu,u) + \kappa^2(Tu,u) 
	= \mathbf{v}^T\bigl(\mathbf{A}\mathbf{T} + \kappa^2\mathbf{S}\mathbf{T}\bigr)\mathbf{u} \\[1ex]
	&= \mathbf{v}^T\bigl(\mathbf{A} + \kappa^2\mathbf{S}\bigr)\mathbf{M}^{-1}_{AS,2}\mathbf{B}\mathbf{u} 
	= \mathbf{v}^T\mathbf{D}_\kappa\mathbf{M}^{-1}_{AS,2}\mathbf{B}\mathbf{u} 
	= \langle\mathbf{M}^{-1}_{AS,2}\mathbf{B}\mathbf{u},\mathbf{v}\rangle_{\mathbf{D}_\kappa}.
\end{align*}
\end{proof}


\subsection[The $H_k$-GenEO coarse space]{The {\boldmath$H_k$}-GenEO coarse space}

In order to use the $H_k$-GenEO coarse space, we need to recall some definitions from \cite{Spillane:2014:ARC}.

\begin{definition}[\cite{Spillane:2014:ARC} Definition 3.2] \label{Def: 3_2}
Given a subdomain $D$, which is formed from a union of elements from the whole domain, let 
\begin{equation*}
	\ovdof(D) \assign \bigl\{\ell \mid 1 \le \ell \le n \text{ and } \supp(\phi_\ell) \cap D \ne \emptyset\bigr\}
\end{equation*}
denote the set of degrees of freedom that are active in the subdomain $D$, including ones on the boundary.
In a similar manner, let
\begin{equation*}
	\dof(D) \assign \bigl\{\ell \mid 1 \le \ell \le n \text{ and } \supp(\phi_\ell) \subset \overline{D}\bigr\}
\end{equation*}
denote the internal degrees of freedom.
\end{definition}

\begin{definition}[Partition of unity] \label{def: 2_7}
	Let $D_i$ be an overlapping subdomain for $i = 1, \ldots, L$, and let $\dof(D_i)$ be as in Definition~\ref{Def: 3_2}.
	For any degree of freedom, $l\in\{1,\ldots,n\}$, let $\mu_l$ denote the number of subdomains for which $l$ is an internal degree of freedom, i.e.\
	\begin{equation}
		\mu_l \assign \#\bigl\{i \mid 1 \le i \le L,\, l \in \dof(D_i)\bigr\}
	\end{equation}
	Then the local partition of unity operator, $\Xi_D: \widetilde{V}_D \rightarrow V_D$ is defined by
	\begin{equation*}
		\Xi_D(v) \assign \sum_{j\in\dof(D_i)}\frac{1}{\mu_l }v_l \phi_l^i \qquad 
		\text{for} \ v = \sum_{l\in\ovdof(D_i)}v_l\phi_l^i \in \widetilde{V}_D.
	\end{equation*}
\end{definition}

The local generalised eigenvalue problem, solved on the coarse decomposition, that is going to form the basis of the coarse space can now be introduced. 

\begin{definition} \label{defH_GenEO}
For each $i\in\{1,\ldots,N\}$ we define the following generalised eigenvalue problem. 
Find $p^i_\ell\in\widetilde{V}_i^c\setminus\{0\}$ and $\lambda_\ell^i\in\mathbb{R}$ such that
\begin{equation}\label{eq: 5_12} 
   b_{\Omega_i^c}(p^i_\ell,v) = \lambda_\ell^i\bigl(\Xi_i^c(p^i_\ell),\Xi_i^c(v)\bigr)_{1,\kappa,\Omega_i^c} \qquad 
   \text{for all} \ v\in\widetilde{V}_i^c. 
\end{equation}
\end{definition}



\begin{definition}[$H_\kappa$-GenEO Coarse space] \label{def: 5_1}
Let $(p_\ell^i,\lambda_\ell^j)$ be as in Definition~\ref{defH_GenEO} and
and let $m_i$ be such that $\lambda_{m_i+1}^i>0$ for every $i\in\{1,\ldots,N\}$.
The coarse space, $V_0$, is given by 
\begin{equation}
	V_0 \assign \linspan \bigl\{E_i^c\Xi_i^c(p_\ell^i) \mid l = 1, \ldots, m_i \text{ and } i = 1, \ldots, N\bigr\}.
\end{equation}
\end{definition}

We will also need the following notations
\begin{equation}
	\Lambda \assign \max_{T\in\mathcal{T}_h}\bigl(\#\bigl\{\Omega^f_{i,j} \mid 1 \le i \le N, 1 \le j \le Q,  T \subset \Omega^f_{i,j}\bigr\}\bigr), 
	\qquad
	\tau \assign \min_{1 \le i \le N} \lambda^i_{m_i+1}. 
\end{equation}

\begin{remark}
    The key difference between this generalised eigenvalue problem and the one used in \textup{\cite{Spillane:2014:ARC}} and in the more recent work \textup{\cite{Bootland:2022:OSM}} is that this definition has an indefinite left-hand side, meaning that the standard spectral theory used for the eigenvalue problem is no longer applicable. Secondly the right-hand side in Definition~\ref{defH_GenEO} is based on a $k$-weighted scalar product. As we shall see, the use of a weighted scalar product and the induced norm is key in getting rigorous $k$-dependent estimates. Another key difference is that we are using two different sizes of domain decomposition for the one-level and coarse space.
\end{remark}

%% file: theoretical_tools.tex
\section{Statement of the main result and theoretical tools}
\label{sec:theory}

We start this section with the statement of the main result and we continue by providing a few technical lemmas needed 
in the proof of this result, which is given in the next section. 

\begin{theorem}[GMRES convergence of the two-level preconditioned system] \label{theorem: convergence}
    Let Assumptions \ref{ass: 2_1} and \ref{Ass: 2_3} be satisfied and let $k>0$.
    Then there exists $h_1>0$ such the following statements are true for all $h\in(0,h_1)$. 
    Let $H^f$ and $\tau$ be such that, with $\Theta\assign\tau^{-1}$,
    \begin{equation}
	\label{eq: 4_3}
    	s \assign 2\Lambda^2 \bigl(2+3\Lambda^4 \Theta\bigr)\Bigl(2k\Theta^{\frac{1}{2}}(1+\Cstab)+3kH^f\Bigr) < 1.
    \end{equation}
    When GMRES is applied with the $\langle\cdot,\cdot\rangle_{\mathbf{D}_k}$-inner product with $\mathbf{D}_\kappa$ as in \eqref{def:Dk} 
    to solve the preconditioned system given by \eqref{eq: 2_24}, then after $m$ iterations, the norm of the residual, $\mathbf{r}^{(m)}$, 
    is bounded as follows:
	\begin{equation}\label{eq: 4_22}
		\|\mathbf{r}^{(m)}\|_{\mathbf{D}_{\kappa}}^2 \le \bigl(1-\gamma^2\bigr)^m \|\mathbf{r}^{(0)}\|_{\mathbf{D}_{\kappa}}^2, 
	\end{equation} 
    where $\gamma$ is given by
    \begin{equation}\label{def_gamma}
    	\gamma \assign \frac{1-s}{(2+3\Lambda^4 \Theta)(18+8\Lambda^3)}\,.
    \end{equation}
\end{theorem}

\begin{corollary} \label{conditions_corollary}
	Assume that $k\ge1$.  If \eqref{eq: 4_3} holds, then there exists $C>0$ such that
	\begin{equation}\label{eq:main}
		H^f k \le C \qquad\text{and}\qquad (1+\Cstab)^2k^2 \le C\tau.
	\end{equation}
	Conversely, if \eqref{eq:main} holds for $C>0$ small enough so that
	\begin{equation}\label{C_small}
		2\Lambda^2 \bigl(2+3\Lambda^4 C\bigr)\bigl(2\sqrt{C}+3C\bigr) < 1,
	\end{equation}
	then \eqref{eq: 4_3} is satisfied and $\gamma$ from \eqref{def_gamma} is bounded below by a positive
	constant that depends only on $\Lambda$ and $C$ but not on the coefficient $A$ and the wave-number $k$,
	which leads to robust GMRES convergence.
%
\end{corollary}

\begin{proof}
	Assume first that \eqref{eq: 4_3} is satisfied.  Since $\Lambda\ge1$, we have
	\[
		2k\Theta^{\frac{1}{2}}(1+\Cstab)+3kH^f \le \frac{1}{4}
	\]
	and hence
	\[
		kH^f \le \frac{1}{12} \qquad\text{and}\qquad
		\frac{(1+\Cstab)^2 k^2}{\tau} = (1+\Cstab)^2 k^2\Theta \le \frac{1}{64}\,,
	\]
	which shows that \eqref{eq:main} holds.
	
	Conversely, assume that \eqref{eq:main} is satisfied with $C>0$ such that \eqref{C_small} holds.
	Since, by assumption, $k\ge1$, we have
	\[
		\Theta = \frac{1}{\tau} \le \frac{C}{(1+\Cstab)^2 k^2} \le C.
	\]
	Hence
	\begin{align*}
		s &= 2\Lambda^2 \bigl(2+3\Lambda^4\Theta\bigr)\Bigl(2k\Theta^{\frac{1}{2}}(1+\Cstab)+3kH^f\Bigr)
		\le 2\Lambda^2 \bigl(2+3\Lambda^4 C\bigr)\bigl(2\sqrt{C}+3C\bigr) < 1,
	\end{align*}
	i.e.\ \eqref{eq: 4_3} is satisfied.
	Moreover, we obtain
	\[
		\gamma = \frac{1-s}{(2+3\Lambda^4\Theta)(18+8\Lambda^3)}
		\ge \frac{1-2\Lambda^2\bigl(2+3\Lambda^4 C\bigr)\bigl(2\sqrt{C}+3C\bigr)}{(2+3\Lambda^4C)(18+8\Lambda^3)},
	\]
	which gives a lower bound for $\gamma$ that depends only on $\Lambda$ and $C$.
%
\end{proof}

In practice, conditions \eqref{eq:main} will introduce a constraint in the size of the subdomains, 
depending on $k$ and on the number of modes to be added in the coarse space.

\subsection[Properties of the $H_k$-GenEO coarse space]{Properties of the {\boldmath$H_k$}-GenEO coarse space}
The first lemma is an adaptation of \cite[Lemma~3.1]{Bootland:2022:OSM} and \cite[Lemma~2.11]{Spillane:2014:ARC} 
to the case of an indefinite problem.  It gives an error estimate for the local projection operator, 
which is used to approximate a function $v \in \widetilde{V}_i^c$ in the space being spanned by the eigenfunctions from 
Definition~\ref{defH_GenEO}.  We see that, despite the indefinite nature of the bilinear form $b$, similar results hold.

\begin{lemma}
\label{Lemma_3_1}
Let $i \in \lbrace 1, \ldots ,N \rbrace$ and let $\lbrace (p_\ell^i , \lambda _\ell^i) \rbrace$ be the eigenpairs of the generalised eigenproblem, 
given in Definition~\ref{defH_GenEO}. Using the results of Lemma \ref{lem_ind_ev}, suppose that $m_i \in \{r+1,\ldots,n-s\}$ is such that $0 < \lambda _{m_i + 1}^i < \infty$. 
Then the local projector, $\Pi_{m_i}^c$, defined by
\begin{equation}
\label{eq:proj}
	\Pi_{m_i}^c v \assign -\sum_{l=1}^r b_{\Omega_i^c}(v,p_l)p_l + \sum_{l=r+1}^{m_i} \left( \Xi_i^c(v), \Xi_i^c(p_l) \right)_{1,\kappa,\Omega_i^c}p_l,
	\qquad v\in\widetilde{V}_i^c,
\end{equation}
satisfies 
\begin{equation}
\label{new_1}
	0 \le b_{\Omega_i^c}(w,w) \le \|v\|^2_{1,\kappa,\Omega_i^c} 
	\qquad\text{and}\qquad
	\|\Xi_i^c(w)\|^2_{1,\kappa,\Omega_i^c} \le \frac{1}{\lambda_{m_i+1}^i} b_{\Omega_i^c}(w,w) \qquad 
	\text{for all} \ v \in \widetilde{V}_i^c,
\end{equation}
where $w = v-\Pi_{m_i}^c v $.
\end{lemma} 

\begin{proof}
The second inequality of \eqref{new_1} can be deduced from the results of Lemma \ref{lem_abs_proj} by setting $\texttt{b} \mapsto b_{\Omega_i^c}$, $c \mapsto \left( \Xi_i^c(\cdot), \Xi_i^c(\cdot) \right)_{1,\kappa,\Omega_i^c}$ and $m \mapsto m_i$. In order to prove the first inequality of \eqref{new_1} by using the projector property we see that
\begin{align*}
b_{\Omega_i^c}\!\left( v,v \right) &= b_{\Omega_i^c}\!\left( w + \Pi_{m_i}^c v, w + \Pi_{m_i}^c v \right)
= b_{\Omega_i^c}\!\left( w , w \right)  + b_{\Omega_i^c}\!\left( \Pi_{m_i}^c v, \Pi_{m_i}^c v \right).
\end{align*}
Taking into account that $b_{\Omega_i^c}(v,v) = \| v \|^2_{1, \kappa, {\Omega_i^c}} - 2 \kappa^2 \|v\|^2_{\Omega_i^c}$ we have
\begin{align*}
	b_{\Omega_i}(w,w) &= b_{\Omega_i^c}\!\left( v,v \right)  - b_{\Omega_i^c}\!\left( \Pi_{m_i}^c v, \Pi_{m_i}^c v \right) 
	=  \|v\|^2_{1,\kappa,\Omega_i^c} - 2 \kappa^2 \|v\|^2_{\Omega_i^c} - ( \| \Pi_{m_i}^c v\|^2_{1, \kappa, {\Omega_i^c}} - 2 \kappa^2 \|\Pi_{m_i}^c v\|^2_{\Omega_i^c} )
	\\[1ex]
	&= \|v\|^2_{1,\kappa,\Omega_i^c} - \| \Pi_{m_i}^c v\|^2_{1, \kappa, {\Omega_i^c}} -2k^2(\| v \|^2_{\Omega_i^c} -\| \Pi_{m_i}^c v \|^2_{\Omega_i^c}) \le \| v \|^2_{1, \kappa, {\Omega_i^c}},
\end{align*}
the last inequality results from the projector property $\| v \|_{\Omega_i^c} \ge  \| \Pi_{m_i}^c v \|_{\Omega_i}$.
\end{proof}

From these local error estimates, it is possible to build a global approximation property. 

\begin{lemma}[Global approximation property]
\label{Lemma_3_2}
Let $\lambda_\ell^i$, $m_i$ and $\Pi_{m_i}^c$ be as in Lemma~\ref{Lemma_3_1} and set
\[
	\Theta \assign \max_{1\le i\le N}\frac{1}{\lambda_{m_i+1}^i}.
\]  
Further, let $v \in V^h$ and define
\begin{equation}\label{def_z_0_z_i}
	z_0 \assign \sum_{i=1}^N \sum_{j=1}^{Q_i} E^f_{i,j} \Xi^f_{i,j} R^f_{i,j} \bigl(\Pi_{m_i}^c {v|_{\Omega_i^c}}\bigr) 
	\quad\text{and}\quad 
	z^f_{i,j} \assign \Xi^f_{i,j} R^f_{i,j} \bigl({v|_{\Omega_i^c}}-\Pi_{m_i}^c{v|_{\Omega_i^c}}\bigr),
\end{equation}
for $i = 1, \ldots ,N$ and $j = 1, \ldots ,Q$.
Then $z^f_{i,j}\in V^f_{i,j}$ for $i\in\{1,\ldots,N\}$ and $j\in\{1,\ldots,Q\}$,
\begin{equation}\label{inequ_sum_z_i}
	 \sum_{i=1}^N \sum_{j=1}^{Q_i} \| z^f_{i,j} \|_{1,k,\Omega^f_{i,j}}^2 \le \Lambda^2 \Theta\|v\|_{1,k}^2
\end{equation}
and
\begin{equation}\label{inequ_v_z_0}
	\inf_{z\in V_0}\|v-z\|_{1,k}^2 \le \|v-z_0\|_{1,k}^2 \le \Lambda^4\Theta\|v\|^2_{1,k}, 
\end{equation}
\end{lemma}

\begin{proof}
Using Lemma~\ref{Lemma_3_1} and the observation that

\begin{equation*}
      \sum_{i=1}^N  \sum_{j=1}^{Q_i}  \left\| z^f_{i,j} \right\|_{1,k, \Omega^f_{i,j}}^2 \le \Lambda \sum_{i=1}^N \left\| \Xi_i^c \bigl({v|_{\Omega_i^c}}-\Pi_{m_i}^c{v|_{\Omega_i^c}}\bigr) \right\|_{1,k, \Omega_i}^2,
\end{equation*}

 we obtain
\begin{align}
	 \sum_{i=1}^N \sum_{j=1}^{Q_i} \| z^f_{i,j} \|_{1,k,\Omega_i}^2 
	&\le  \Lambda \sum_{i=1}^N\frac{1}{\lambda_{m_i+1}^i}
	b_{\Omega_i^c}\Bigl(\bigl({v|_{\Omega_i^c}}-\Pi_{m_i}^c{v|_{\Omega_i^c}}\bigr),\bigl({v|_{\Omega_i^c}}-\Pi_{m_i}^c{v|_{\Omega_i^c}}\bigr)\Bigr)
	\nonumber
	\\[1ex]
	&\le \Lambda \sum_{i=1}^N\frac{1}{\lambda_{m_i+1}^i}\big\|v|_{\Omega_i^c}\big\|_{1,k,\Omega_i^c}^2
	\le \Lambda  \Theta\sum_{i=1}^N\big\|v|_{\Omega_i^c}\big\|_{1,k,\Omega_i^c}^2
	\le \Lambda^2 \Theta\|v\|_{1,k}^2,
	\label{proof_inequ_sum_z_i}
\end{align}
which shows \eqref{inequ_sum_z_i}.
From this we can deduce that
\begin{align*}
    \|v-z_0\|_{1,\kappa}^2 &= \Bigg\|\sum_{i=1}^N \sum_{j=1}^{Q_i} E^f_{i,j} \Xi^f_{i,j} R^f_{i,j} \bigl(v|_{\Omega_i^c}\bigr)-\sum_{i=1}^N \sum_{j=1}^{Q_i} E^f_{i,j} \Xi^f_{i,j} R^f_{i,j} \bigl(\Pi_{m_i}^c {v|_{\Omega_i^c}}\bigr)\Bigg\|_{1,k}^2 \\
    &= \Bigg\|\sum_{i=1}^N \sum_{j=1}^{Q_i} E^f_{i,j} z^f_{i,j}\Bigg\|_{1,k}^2 \le \Lambda^2 \sum_{i=1}^N \sum_{j=1}^{Q_i}  \| z^f_{i,j} \|^2_{1, k, \Omega_i},
\end{align*}
which, together with \eqref{proof_inequ_sum_z_i}, proves the second inequality in \eqref{inequ_v_z_0}.
Since $z_0\in V_0$, also the first inequality in \eqref{inequ_v_z_0} follows.
%
\end{proof}

The following lemma shows that the GenEO coarse space, in combination with the local finite element space, allows for a stable decomposition. 
This property is a key component for bounding the condition number in the two-level Schwarz  preconditioner for positive definite cases, 
as used in \cite{Spillane:2014:ARC}.  It will be used for the same purpose in this work for the indefinite case. 

\begin{lemma}[Stable decomposition]
\label{lemma_3_3}
Let $v \in V^h$ and define $z^f_{i,j}$, $i=1,\ldots,Q$, as in \eqref{def_z_0_z_i}.
Then $v = z_0 + \sum_{i=1}^N \sum_{j=1}^{Q_i} E^f_{i,j} z^f_{i,j}$ and
\begin{equation*}
	\|z_0\|_{1,k}^2 + \sum_{i=1}^N \sum_{j=1}^{Q_i} \|z^f_{i,j}\|_{1,k,\Omega^f_{i,j}}^2 \le \bigl(2+3\Lambda^4 \Theta\bigr)\|v\|^2_{1,k}. 
\end{equation*}
\end{lemma}

\begin{proof}
The relation $v = z_0 + \sum_{i=1}^N \sum_{j=1}^{Q_i} E^f_{i,j} z^f_{i,j}$ follows easily from the definition of the partition of unity.
From \eqref{inequ_v_z_0} we obtain
\[
	\|z_0\|_{1,k}^2 \le \Bigl(\|z_0-v\|_{1,k}+\|v\|_{1,k}\Bigr)^2
	\le 2\Bigl(\|z_0-v\|_{1,k}^2+\|v\|_{1,k}^2\Bigr)
	\le 2\bigl(1+\Lambda^4\Theta\bigr)\|v\|_{1,k}^2.
\]
Together with \eqref{inequ_sum_z_i} we arrive at
\[
	\|z_0\|_{1,k}^2 + \sum_{i=1}^N \sum_{j=1}^{Q_i} \|z^f_{i,j}\|_{1,k,\Omega_i^c}^2
	\le 2\bigl(1+\Lambda^4\Theta\bigr)\|v\|_{1,k}^2 + \Lambda^2\Theta\|v\|_{1,k}^2
	\le \bigl(2+3\Lambda^4\Theta\bigr)\|v\|_{1,k}^2
\]
since $\Lambda\ge1$.
%
\end{proof}

It is convenient to introduce here the projection operators $P^f_{i,j}: V^h \rightarrow V^f_{i,j}$, such that for each $i = 1, \ldots, N$, and $j = 1, \ldots, Q$,
\begin{equation}
	(P^f_{i,j} u, v)_{1,k,\Omega^f_{i,j}} =(u, E^f_{i,j} v)_{1,k,\Omega}. \quad \forall v \in V^f_{i,j} \label{eq: 3_10_nested}.
\end{equation}
In \cite[Section~2.2]{Toselli:2005:DDM} it is proved that these operators are well defined and orthogonal projections.
For the projector $P_0 : V^h \rightarrow V^0$, this is defined as
\begin{equation}
	(P_0 u, v)_{1,k,\Omega^f_{i,j}} = (u, v), \quad \forall v \in V_0. \label{eq: 3_10_zero}
\end{equation}
Moreover, we define the operator $P: V^h \rightarrow V^h$ as
\begin{equation}
	P \assign E_0P_0 + \sum_{i=1}^N \sum_{j=1}^{Q_i} E^f_{i,j} P^f_{i,j}. \label{eq: 3_11}
\end{equation} 

\begin{proposition}
	Under the same assumptions as in  Lemma~\ref{lemma_3_3}, any $u \in V^h$ satisfies
	\begin{equation}\label{eq: 3_12}
		\|u\|^2_{1,k} \le \bigl(2+3\Lambda^4\Theta\bigr)(Pu,u)_{1,k}. 
	\end{equation}
\end{proposition}

\begin{proof}
Using the definition of $P_i$, the Cauchy--Schwarz inequality and Lemma~\ref{lemma_3_3} we get
\begin{align*}
	\|u\|_{1,k}^2 &= \Biggl(u, z_0 + \sum_{i=1}^N \sum_{j=1}^{Q_i} E^f_{i,j} z^f_{i,j}\Biggr)_{1,k} 
	= (u, z_0)_{1,k} + \Biggl(u,\sum_{i=1}^N \sum_{j=1}^{Q_i} E^f_{i,j} z^f_{i,j}\Biggr)_{1,k} \\ 
    &= (P_0u, z_0)_{1,k} + \sum_{i=1}^N \sum_{j=1}^{Q_i}( P^f_{i,j} u, z^f_{i,j})_{1,k, \Omega^f_{i,j}} 
    \le \| P_0u\|_{1,k} \|z_0\|_{1,k} + \sum_{i=1}^N \sum_{j=1}^{Q_i}\| P^f_{i,j} u \|_{1,k, \Omega^f_{i,j}} \|z^f_{i,j}\|_{1,k, \Omega^f_{i,j}} \\[1ex]
    &\le \left(\| P_0u\|_{1,k}^2 + \sum_{i=1}^N \sum_{j=1}^{Q_i}\| P^f_{i,j} u \|_{1,k, \Omega^f_{i,j}}^2 \right)^\frac{1}{2} (2+3\Lambda^4\Theta\bigr)^\frac{1}{2}\|u\|_{1,k} \\[1ex]
	&= \left(( E_0 P_0u, u)_{1,k} + \sum_{i=1}^N \sum_{j=1}^{Q_i} ( E^f_{i,j} P^f_{i,j} u, u)_{1,k} \right)^\frac{1}{2} (2+3\Lambda^4\Theta\bigr)^\frac{1}{2}\|u\|_{1,k} \\[1ex]
    &= \left(( E_0 P_0u, u)_{1,k} +  \left( \sum_{i=1}^N \sum_{j=1}^{Q_i}E^f_{i,j} P^f_{i,j} u, u\right)_{1,k} \right)^\frac{1}{2} (2+3\Lambda^4\Theta\bigr)^\frac{1}{2}\|u\|_{1,k} \\[1ex]
    &= \left( E_0 P_0 u + \sum_{i=1}^N \sum_{j=1}^{Q_i}E^f_{i,j} P^f_{i,j} u, u \right)_{1,k} ^\frac{1}{2} (2+3\Lambda^4\Theta\bigr)^\frac{1}{2}\|u\|_{1,k} = ( P u,u)_{1,k}^\frac{1}{2}\bigl(2+3\Lambda^4\Theta\bigr)^\frac{1}{2}\|u\|_{1,k} 
\end{align*}
and the inequality in \eqref{eq: 3_12} follows.
\end{proof}


\subsection[Existence and stability of $T^f_{i,j}$, $i=1,\ldots,N$]{Solvability and stability of {\boldmath$T^f_{i,j}$}, {\boldmath$i=1,\ldots,N$} and {\boldmath$j=1,\ldots,Q$}}

The next lemma shows that $T^f_{i,j}$ is well defined for $i = 1, \ldots, N$ and $j = 1, \ldots, Q$, if the subdomains are small
enough with respect to $k$.

\begin{lemma}[$T^f_{i,j}$ is well defined for $j = 1, \ldots, Q$]
\label{lemma_3_4}
If $H^f k < \sqrt{2}$, then $b_{\Omega^f_{i,j}}(\cdot,\cdot)$ is positive definite and the operators $T^f_{i,j}$, $j=1,\ldots,Q$ are well defined.
\end{lemma}

\begin{proof}
Using the definition of $b_{\Omega^f_{i,j}}(\cdot,\cdot)$ and the Friedrichs inequality we obtain
\begin{equation*}
	b_{\Omega^f_{i,j}}(u,u) = a_{\Omega^f_{i,j}}(u,u) - k^2(u,u)_{\Omega^f_{i,j}} 
	\ge \frac{2}{{H^f}^2}\|u\|_{\Omega^f_{i,j}}^2 - k^2\|u\|_{\Omega^f_{i,j}}^2 
	= \frac{2-{H^f}^2k^2}{{H^f}^2}\|u\|_{\Omega^f_{i,j}}^2,
\end{equation*}
and hence $b_{\Omega^f_{i,j}}(\cdot,\cdot)$ is positive definite.  By the Lax--Milgram Lemma, this ensures that $T^f_{i,j}$ is well defined.  
\end{proof}

\begin{remark}
	Whilst this a sufficient condition for $T^f_{i,j}$ to be well defined, it is not a necessary condition. 
	In general, if $k$ is different from an eigenvalue of $a_{\Omega^f_{i,j}}$, then $T^f_{i,j}$ will be well defined. 
\end{remark}

Before we proceed with the stability of $T_{i,j}^f$, we formulate useful estimates for the forms $b$ and $b_{\Omega_i}$.

\begin{lemma}\label{lem:est_b}
	The following estimates are valid:
	\begin{alignat*}{2}
		|b(u,v)| &\le \|u\|_{1,k}\|v\|_{1,k} \qquad && \text{for} \ u,v\in V^h,
		\\[1ex]
		|b_{D}(u,v)| &\le \|u\|_{1,k,D}\|v\|_{1,k,D} \qquad && 
		\text{for} \,u,v,\in\widetilde V_D.
	\end{alignat*}
\end{lemma} 

\begin{proof}
We only prove the first inequality; the proof of the second one is analogous.
Let $u,v\in V^h$.  Using Cauchy--Schwarz inequalities we obtain
\begin{align*}
	|b(u,v)| &= |a(u,v)-k^2(u,v)|
	\le |a(u,v)| + k^2|(u,v)|
	\le \|u\|_a\|v\|_a + k\|u\|\,k\|v\|
	\\[0.5ex]
	&\le \sqrt{\|u\|_a^2+k^2\|u\|^2}\cdot\sqrt{\|v\|_a^2+k^2\|v\|^2}
	= \|u\|_{1,k}\|v\|_{1,k}.
\end{align*}
\end{proof}

In order to show the robustness of the additive Schwarz method, it is necessary to prove the stability estimates 
for the operators $T^f_{i,j}$.

\begin{lemma}[Stability of $T^f_{i,j}$, $j=1,\ldots,Q$]
\label{lemma_3_6}
Suppose that $H^f\kappa \le \frac{\sqrt{2}\,}{2}$.  Then, for all $u \in V^h$,
\begin{equation}\label{eq: 3_31}
	\|T^f_{i,j} u\|_{1,k,\Omega^f_{i,j}} \le 2\big\|u|_{\Omega^f_{i,j}}\big\|_{1,k,\Omega^f_{i,j}}.
\end{equation}
\end{lemma}

\begin{proof}
	Using the definition of $T^f_{i,j}$ in \eqref{eq: bT}, Lemma~\ref{lem:est_b} and the Friedrichs inequality \eqref{eq: 2_16} we obtain
	\begin{align*}
		\|T^f_{i,j} u\|^2_{1,k,\Omega^f_{i,j}} 
		&= b_{\Omega^f_{i,j}}(T^f_{i,j} u,T^f_{i,j} u) + 2k^2\|T^f_{i,j} u\|_{\Omega_i} 
		= b_{\Omega_i}\bigl(u|_{\Omega^f_{i,j}},T^f_{i,j} u\bigr) + 2k^2\|T^f_{i,j} u\|_{\Omega^f_{i,j}}
		\\[1ex]
		&\le \big\|u|_{\Omega^f_{i,j}}\big\|_{1,k,\Omega^f_{i,j}}\|T^f_{i,j} u\|_{1,k,\Omega^f_{i,j}} + 2k^2\|T^f_{i,j} u\|_{\Omega^f_{i,j}}
		\\[1ex]
		&\le \big\|u|_{\Omega^f_{i,j}}\big\|_{1,k,\Omega^f_{i,j}}\|T^f_{i,j} u\|_{1,k,\Omega^f_{i,j}} + k^2(H^f)^2\|T^f_{i,j} u\|^2_{1,k,\Omega^f_{i,j}},
	\end{align*}
	which implies that
	\begin{equation*}
		\bigl(1-k^2 (H^f)^2\bigr)\|T^f_{i,j} u\|_{1,k,\Omega^f_{i,j}} \le \big\|u|_{\Omega^f_{i,j}}\big\|_{1,k,\Omega^f_{i,j}}.
	\end{equation*}
	Using the condition $H^f \kappa\le\frac{\sqrt{2}\,}{2}$ we can deduce the required bound.
\end{proof}

\subsection[Existence and stability of $T_0$]{Existence and stability of {\boldmath$T_0$}}
To ensure that $T_0$ is well defined, a condition on $\Theta$ is required.

\begin{lemma}[$T_0$ is well defined] \label{lemma_3_5}
	Suppose that 
	\begin{equation}\label{3_7_a}
		\sqrt{2}\,k\Lambda^2\Theta^\frac{1}{2}(1+\Cstab) < 1. 
	\end{equation}
	Then there exists $h_1 > 0$ such that, for all $h\in(0,h_1)$, the operator $T_0$ is well defined.
\end{lemma}

\begin{proof}
	Assume that there exists an $w_0 \in V_0\backslash \lbrace 0 \rbrace$ such that
	\begin{equation}\label{eq:w_0_sol}
		b(w_0,z) = 0 \qquad \text{for all} \ z \in V_0.
	\end{equation}
	Let $w \in H^1_0(\Omega)$ be the solution of 
	\begin{equation*}
		b(w,v)=(w_0,v) \qquad \text{for all} \ v \in H^1_0(\Omega).
	\end{equation*}
	Lemma~\ref{schatz_wang} implies that there exists an $h_1 > 0$ such that, for all $h \in (0, h_1)$, 
	there is a solution $w_h \in V^h$ of
	\begin{equation*}
		b(w_h, v) = (w_0,v) \qquad \text{for all} \ v \in V^h.
	\end{equation*}
	This holds, in particular, for $v=w_0$.  This, together with Lemma~\ref{lem:est_b}, yields that, for all $z \in V_0$,
	\begin{equation*}
		\|w_0\|^2 = b(w_h, w_0) = b(w_h, w_0) - b(z, w_0) = b(w_h - z, w_0)
		\le \|w_h - z\|_{1,k} \|w_0\|_{1,k}.
	\end{equation*}
	As this is true for all $z \in V_0$, we obtain
	\begin{equation}\label{eq: 3_5_f_sq}
		\|w_0\|^2 \le \|w_0\|_{1,k}\inf_{z \in V_0}\|w_h-z\|_{1,k}
		\le \|w_0\|_{1,k}\Lambda\Theta^{\frac{1}{2}}\|w_h\|_{1,k}
	\end{equation}
	by Lemma~\ref{Lemma_3_2}.  Equation \eqref{eq:w_0_sol} with $v=w_0$ implies that
	\begin{equation*}
		0 = b(w_0,w_0) = \|w_0\|_{1,k}^2-2k^2\|w_0\|^2,
	\end{equation*}
	which, together with \eqref{eq: 3_5_f_sq}, yields
	\begin{equation}\label{eq: 3_5_1} 
		\|w_0\|^2 \le \sqrt{2}\,k\|w_0\|\Lambda^2\Theta^\frac{1}{2}\|w_h\|_{1,k}. 
	\end{equation}
	It is possible to combine \eqref{eq: 2_10} and \eqref{eq: 2_13} to obtain
	\begin{align}
		\|w_h\|_{1,k} &\le \|w\|_{1,k} + \|w-w_h\|_{1,k}
		= \|w\|_{1,k} + \sqrt{\|w-w_h\|_a^2 + k^2\|w-w_h\|\|^2}
		\nonumber\\[0.5ex]
		&\le \Cstab\|w_0\| + \sqrt{\amax\|w-w_h\|_{H^1(\Omega)}^2 + k^2\|w-w_h\|^2}
		\nonumber\\[0.5ex]
		&\le \Cstab\|w_0\| + \sqrt{\amax\varepsilon^2\|w_0\|^2+k^2\varepsilon^4\|w_0\|^2}
		\nonumber\\[0.5ex]
		&\le (\Cstab+1)\|w_0\|
		\label{eq: w_h_comb}
	\end{align}
	if $\varepsilon$ is chosen small enough.
	With \eqref{eq: w_h_comb} used in \eqref{eq: 3_5_1}, we arrive at
	\begin{equation*}
        \|w_0\|^2 \le \sqrt{2}\,k\Lambda^2\Theta^\frac{1}{2}(1+\Cstab)\|w_0\|^2.
    \end{equation*}
	Together with the condition in \eqref{3_7_a}, this leads to a contradiction.
\end{proof}

Another requirement is to find the stability conditions for $T_0$, which is done in the next lemma.

\begin{lemma}[Stability of $T_0$]
\label{lemma: 3_7}
Assume that \eqref{3_7_a} is satisfied.  Then there exists $h_1>0$ such that, for $h\in(0,h_1)$,
\begin{equation}
	\|T_0 u - u\| \le  \Lambda^2\Theta^\frac{1}{2}(1 + \Cstab)\|T_0 u - u\|_{1,k} \qquad \text{for all} \ u \in V_h. 
\label{3_7_b}
\end{equation}
Suppose, in addition that, 
\begin{equation}\label{3_7_c}
    2k\Lambda^2\Theta^\frac{1}{2} (1 + \Cstab) \le \frac{1}{2};
\end{equation} 
then 
\begin{equation}\label{3_7_d}
	\|u - T_0 u\|_{1,k} \le 2 \|u\|_{1,k} \qquad \text{for all} \ u \in V^h.
\end{equation}
\end{lemma}

\begin{proof}
Under condition \eqref{3_7_a}, Lemma~\ref{lemma_3_5} ensures the existence of $h_1>0$ such that, for $h\in(0,h_1)$, 
the operator $T_0: V^h \rightarrow V_0$ is well defined.  Consider the following auxiliary problem, 
\begin{equation}\label{eq 3_7_4}
	\text{find } w_h \in V^h  \text{ such that} \quad b(w_h, v) = \bigl(T_0 u - u,v\bigr) \quad \text{for all} \ v \in V^h. 
\end{equation}
Note that, by the definition of $T_0$, we have $b(T_0 u-u,z)=0$ for all $z \in V_0$.
We can use $v = T_0 u - u \in V^h$ in \eqref{eq 3_7_4} and Lemma~\ref{lem:est_b} to obtain, for every $z\in V_0$,
\begin{align*}
	\|T_0 u - u\|^2 &= b(w_h, T_0 u - u)= b(w_h, T_0 u - u) - b(z, T_0 u - u)
	\\[0.5ex]
	&= b(w_h - z, T_0 u - u)
	\le \|w_h-z\|_{1,k}\,\|T_0 u - u\|_{1,k}.
\end{align*}
Combining this with Lemma~\ref{Lemma_3_2} we obtain
\begin{equation}\label{ineq:T0u-u}
	\|T_0 u - u\|^2 \le \|T_0 u - u\|_{1,k}\inf_{z \in V_0}\|w_h-z\|_{1,k} 
	\le \|T_0 u - u\|_{1,k}\Lambda^2\Theta^\frac{1}{2}\|w_h\|_{1,k}.
\end{equation}
In a similar way as in the proof of Lemma~\ref{lemma_3_5}, one can show that $\|w_h\|_{1,k}\le (1+\Cstab)\|T_0 u - u\|$
since $w_h$ is a solution of \eqref{eq 3_7_4}, which, together with \eqref{ineq:T0u-u}, proves \eqref{3_7_b}.

We now focus on the proof of \eqref{3_7_d}.  By the definition of $P_0$ in \eqref{eq: 3_10_zero} we have
\begin{equation*}
	(P_0 u - u,v)_{1,k} = 0, \qquad u \in V^h, v \in V_0.
\end{equation*}
Using $v = T_0u \in V_0$ in this relation we obtain $(T_0u,P_0 u - u)_{1,k} = 0$.
Since $P_0 u - T_0 u \in V_0$, the definition of $T_0$ yields
\begin{equation}\label{eq:bproj}
	b(u-T_0 u, P_0 u - T_0 u) = 0.
\end{equation}
Now we can use \eqref{eq:bproj}, the link between the bilinear forms $(\cdot,\cdot)_{1,k}$ and $b$, and the Cauchy--Schwarz inequality to obtain
\begin{align*}
	\|u - T_0 u\|_{1,k}^2 &= (u - T_0 u, u - T_0 u)_{1,k} 
	= b(u - T_0 u, u - T_0 u) + 2k^2(u - T_0 u, u - T_0 u) 
	\\[0.5ex]
	&= b(u - T_0 u, u - T_0 u) - b(u - T_0 u, P_0 u - T_0 u) + 2 \kappa^2 (u - T_0 u, u - T_0 u) 
	\\[0.5ex]
	&= b(u - T_0 u, u - P_0 u) + 2k^2(u - T_0 u, u - T_0 u) 
	\\[0.5ex]
	&= (u - T_0 u, u - P_0 u)_{1,k} - 2k^2(u - T_0 u, u - P_0 u) + 2k^2(u - T_0 u, u - T_0 u) 
	\\[0.5ex]
	&= (u - T_0 u, u - P_0 u)_{1,k} + 2k^2(u - T_0 u, P_0 u - T_0 u) 
	\\[0.5ex]
	&\le \|u - T_0 u\|_{1,k}\|u - P_0 u\|_{1,k} + 2k^2\|u - T_0 u\| \|T_0 u - P_0 u\|
	\\[0.5ex]
	&\le \|u - T_0 u\|_{1,k}\|u - P_0 u\|_{1,k} + 2k\|u - T_0 u\| \|T_0 u - P_0 u\|_{1,k}
	\\[0.5ex]
	&= \|u - T_0 u\|_{1,k}\|u - P_0 u\|_{1,k} + 2k\|u - T_0 u\| \|P_0(T_0 u - u)\|_{1,k}
	\\[0.5ex]
	&\le \|u - T_0 u\|_{1,k}\|u\|_{1,k} + 2k\|u - T_0 u\| \|T_0 u - u\|_{1,k},
\end{align*}
where in the last step we used the fact that $P_0$ is an orthogonal projection w.r.t.\ $(\cdot,\cdot)_{1,k}$.
Dividing both sides by $\|u - T_0 u\|_{1,k}$ and using \eqref{3_7_b} and the assumption \eqref{3_7_c} we arrive at
\begin{align*}
	\|u - T_0 u\|_{1,k} &\le \|u\|_{1,k} + 2k\|u - T_0 u\|
	\\[0.5ex]
	&\le \|u\|_{1,k} + 2k\Lambda^2\Theta^{\frac{1}{2}}(1+\Cstab)\|u - T_0 u\|_{1,k}
	\\[0.5ex]
	&\le \|u\|_{1,k} + \frac{1}{2}\|u - T_0 u\|_{1,k},
\end{align*}
which proves \eqref{3_7_d}.
\end{proof}

%% file: main_result.tex
\section{Proof of the main result}

Before Theorem~\ref{theorem: convergence} can be proved, it is necessary to state and prove the following lemma.

\begin{lemma} \label{ther: 4_1}
	Let the assumptions in Theorem~\ref{theorem: convergence} be satisfied and set
	\begin{equation*}
		c_1 \assign \frac{1-s}{2+3\Lambda^4\Theta},
		\qquad
		c_2 \assign 18 + 8\Lambda^3.
	\end{equation*}
	Then, for all $u \in V_h$,
	\begin{equation}
		c_1\|u\|_{1,k}^2 \le (Tu,u)_{1,k},
		\label{eq: 4_4}
	\end{equation}
	and
	\begin{equation}
		\|Tu\|_{1,k}^2 \le c_2 \|u\|_{1,k}^2.
		\label{eq: 4_5}
	\end{equation}
\end{lemma}

\begin{proof}
To prove \eqref{eq: 4_4} and \eqref{eq: 4_5}, let $u \in V^h$.  We proceed in several steps.

\medskip

\noindent 
\textit{Step 1}. 
We start with \eqref{eq: 3_12} and then use \eqref{eq: 3_11}, \eqref{eq: bT}, \eqref{eq: 3_10_nested} and \eqref{eq: 2_26} 
to obtain
\begin{align*}
	&(2+3\Lambda^4\Theta)^{-1}\|u\|_{1,k}^2 \le (Pu,u)_{1,k}
	= (u,Pu)_{1,k}
	= (u,E_0 P_0u + \sum^N_{i=1}\sum^{Q_i}_{j=1} E^f_{i,j} P^f_{i,j}u)_{1,k}
	\\[0.5ex]
    &= (u,E_0 P_0u)_{1,k} + \sum^N_{i=1}\sum^{Q_i}_{j=1}(u,E^f_{i,j} P^f_{i,j}u)_{1,k} \\
	&=  b(u,E_0 P_0u) + 2k^2(u,E_0 P_0u ) +  \sum^N_{i=1}\sum^{Q_i}_{j=1} \left[ b(u, E^f_{i,j} P^f_{i,j}u) + 2k^2(u,E^f_{i,j} P^f_{i,j}u) \right]
	\\[0.5ex]
	&=  b(T_0 u,P_0u) + 2k^2(u,E_0 P_0u ) +  \sum^N_{i=1}\sum^{Q_i}_{j=1} \left[ b_{\Omega^f_{i,j}}(T^f_{i,j} u, P^f_{i,j}u) + 2k^2(u,E^f_{i,j} P^f_{i,j}u) \right]
	\displaybreak[0]\\[0.5ex]
	&=  (T_0 u,P_0u)_{1,k} - 2k^2(T_0 u,P_0u) + 2k^2(u,E_0 P_0u ) + \ldots \\ &\ldots+  \sum^N_{i=1}\sum^{Q_i}_{j=1} \left[ (T^f_{i,j} u, P^f_{i,j}u)_{1,k,\Omega^f_{i,j}} -2k^2(T^f_{i,j} u, P^f_{i,j}u)_{\Omega^f_{i,j}} + 2k^2(u,E^f_{i,j} P^f_{i,j}u) \right]
	\\[0.5ex]
	&=  (E_0 T_0 u,u)_{1,k} - 2k^2(E_0 T_0 u,E_0 P_0u) + 2k^2(u,E_0 P_0u ) + \ldots \\ &\ldots+  \sum^N_{i=1}\sum^{Q_i}_{j=1} \left[ (E^f_{i,j} T^f_{i,j} u, u)_{1,k} -2k^2(E^f_{i,j} T^f_{i,j} u, E^f_{i,j} P^f_{i,j}u) + 2k^2(u,E^f_{i,j} P^f_{i,j}u) \right]
	\\[0.5ex]
    &=  \left[(E_0 T_0 u,u)_{1,k} + \sum^N_{i=1}\sum^{Q_i}_{j=1} (E^f_{i,j} T^f_{i,j} u, u)_{1,k}\right] - \ldots \\ &\ldots- \left( 2k^2(E_0 T_0 u - u,E_0 P_0u ) +  \sum^N_{i=1}\sum^{Q_i}_{j=1} \left[ 2k^2(E^f_{i,j} T^f_{i,j} u - u,E^f_{i,j} P^f_{i,j}u) \right] \right)
    \\[0.5ex]
	&=  \left(E_0 T_0 u + \sum^N_{i=1}\sum^{Q_i}_{j=1} E^f_{i,j} T^f_{i,j} u, u \right)_{1,k} - \left( 2k^2(E_0 T_0 u - u,E_0 P_0u ) +  \sum^N_{i=1}\sum^{Q_i}_{j=1} \left[ 2k^2(E^f_{i,j} T^f_{i,j} u - u,E^f_{i,j} P^f_{i,j}u) \right] \right)
    \\[0.5ex]
    &=  (T u, u)_{1,k} - \left( 2k^2(E_0 T_0 u - u,E_0 P_0u ) +  \sum^N_{i=1}\sum^{Q_i}_{j=1} \left[ 2k^2(E^f_{i,j} T^f_{i,j} u - u,E^f_{i,j} P^f_{i,j}u) \right] \right)
\end{align*}
which can be rewritten as
\begin{equation}\label{eq: 4_9}
	\|u\|^2_{1,k} \le (2+3\Lambda^4\Theta)(T u, u)_{1,k} - (2+3\Lambda^4\Theta)\left( 2k^2(E_0 T_0 u - u,E_0 P_0u ) +  \sum^N_{i=1}\sum^{Q_i}_{j=1} \left[ 2k^2(E^f_{i,j} T^f_{i,j} u - u,E^f_{i,j} P^f_{i,j}u) \right] \right)
\end{equation}

\noindent
\textit{Step 2}. 
The next stage is to bound the sum in \eqref{eq: 4_9}.
We start with the term for $i=0$.  The operator $E_0$ should be considered as a natural embedding like the identity operator 
(and hence $E_i T_0 u = T_0 u$ and $E_i P_0 u = P_0 u$).  
Further, note that the assumption in \eqref{eq: 4_3} implies that \eqref{3_7_c} is satisfied,
and hence we can use \eqref{3_7_b}, \eqref{3_7_d} and the projection property of $P_0$ to obtain
\begin{align}
	-k^2(E_0 T_0 u - u, E_0 P_0 u)
	&\le k^2 \|E_0 T_0 u - u\| \|E_0 P_0 u\| 
	= k^2\| T_0 u -u\| \| P_0 u\|
	\nonumber\\[0.5ex]
	&\le k\|T_0 u -u\| \|P_0 u\|_{1,k} 
	\le k\Lambda^2 \Theta^\frac{1}{2}(1+\Cstab)\|T_0 u - u\|_{1,k}\|P_0 u\|_{1,k}
	\nonumber\\[0.5ex]
	&\le 2k\Lambda^2 \Theta^\frac{1}{2}(1+\Cstab)\|u\|_{1,k}\|P_0 u\|_{1,k}
	\nonumber\\[0.5ex]
	&\le 2k\Lambda^2 \Theta^\frac{1}{2}(1+\Cstab)\|u\|_{1,k}^2.
	\label{eq:10_0}
\end{align}

\noindent
\textit{Step 3}. 
Next we estimate the terms in the sum \eqref{eq: 4_9} for $i\in\{1,\ldots,N\}$, $j\in\{1,\ldots,Q\}$.
The inequality in \eqref{eq: 4_3} implies that the assumption $H^{\alpha}k\le\frac{\sqrt{2}\,}{2}$ in Lemma~\ref{lemma_3_6} is satisfied.
Hence we can use \eqref{eq: 2_16} and \eqref{eq: 3_31} to obtain
\begin{align*}
	-k^2(E^f_{i,j} T^f_{i,j} u-u,E^f_{i,j} P^f_{i,j} u) &= -k^2\bigl(T^f_{i,j} u-u|_{\Omega^f_{i,j}},P^f_{i,j} u\bigr)_{\Omega^f_{i,j}}
	\le k^2 \big\|T^f_{i,j} u-u|_{\Omega^f_{i,j}}\big\|_{\Omega^f_{i,j}}\|P^f_{i,j} u\|_{\Omega^f_{i,j}}
	\\[0.5ex]
	&\le k\bigl\|T^f_{i,j} u-u|_{\Omega^f_{i,j}}\big\|_{1,k,\Omega^f_{i,j}}\|P^f_{i,j} u\|_{\Omega^f_{i,j}}
	\\[0.5ex]
	&\le k\bigl\|T^f_{i,j} u-u|_{\Omega^f_{i,j}}\big\|_{1,k,\Omega^f_{i,j}}\frac{H}{\sqrt{2}\,}\|P^f_{i,j} u\|_{1,k,\Omega^f_{i,j}}
	\\[0.5ex]
	&\le \frac{kH^f}{\sqrt{2}\,}\Bigl(\|T^f_{i,j} u\|_{1,k,\Omega_i}+\big\|u|_{\Omega^f_{i,j}}\big\|_{1,k,\Omega^f_{i,j}}\Bigr)\|P^f_{i,j} u\|_{1,k,\Omega^f_{i,j}}
	\\[0.5ex]
	&\le \frac{3kH^f}{\sqrt{2}\,}\big\|u|_{\Omega^f_{i,j}}\big\|_{1,k,\Omega^f_{i,j}}\|P^f_{i,j} u\|_{1,k,\Omega^f_{i,j}}
	\le 3kH^f\big\|u|_{\Omega^f_{i,j}}\big\|_{1,k,\Omega^f_{i,j}}\|P^f_{i,j} u\|_{1,k,\Omega^f_{i,j}}.
\end{align*}
Taking the double sum over $i\in\{1,\ldots,N\}$ and $j\in\{1,\ldots,Q\}$, applying the Cauchy--Schwarz inequality and the overlap property
we arrive at
\begin{align}
	-k^2\sum_{i=1}^N \sum_{j=1}^{Q_i} (E^f_{i,j} T^f_{i,j} u-u,E^f_{i,j} P^f_{i,j} u) 
	&\le 3kH^f \sum_{i=1}^N \sum_{j=1}^{Q_i} \big\|u|_{\Omega^f_{i,j}}\big\|_{1,k,\Omega^f_{i,j}}\|P^f_{i,j} u\|_{1,k,\Omega^f_{i,j}}
	\nonumber\\[0.5ex]
	&\le 3kH^f \Biggl(\sum_{i=1}^N \sum_{j=1}^{Q_i} \big\|u|_{\Omega^f_{i,j}}\big\|_{1,k,\Omega^f_{i,j}}^2\Biggr)^{\frac{1}{2}}
	\Biggl(\sum_{i=1}^N \sum_{j=1}^{Q_i} \|P^f_{i,j} u\|_{1,k,\Omega^f_{i,j}}^2\Biggr)^{\frac{1}{2}}
	\nonumber\\[0.5ex]
	&\le 3kH^f \Lambda\|u\|_{1,k}\Biggl(\sum_{i=1}^N \sum_{j=1}^{Q_i} \|P^f_{i,j} u\|_{1,k,\Omega^f_{i,j}}^2\Biggr)^{\frac{1}{2}}.
	\label{eq:10_1}
\end{align}
In order to estimate the sum in \eqref{eq:10_1}, we consider
\begin{align*}
	\Bigg\|\sum_{i=1}^N \sum_{j=1}^{Q_i} E^f_{i,j} P^f_{i,j} u\Bigg\|_{1,k}^2
	&\le \Lambda^2 \sum_{i=1}^N \sum_{j=1}^{Q_i} \|E^f_{i,j} P^f_{i,j} u\|_{1,k}^2 = \Lambda^2 \sum_{i=1}^N \sum_{j=1}^{Q_i} (E^f_{i,j} P^f_{i,j} u,u)_{1,k}
	\\[0.5ex]
	&= \Lambda^2 \Biggl(\sum_{i=1}^N \sum_{j=1}^{Q_i} E^f_{i,j} P^f_{i,j} u,\,u\Biggr)_{1,k}
	\le \Lambda^2 \,\Bigg\|\sum_{i=1}^N \sum_{j=1}^{Q_i} E^f_{i,j} P^f_{i,j} u\Bigg\|_{1,k} \|u\|_{1,k},
\end{align*}
and hence
\[
	\Bigg\|\sum_{i=1}^N \sum_{j=1}^{Q_i} E^f_{i,j} P^f_{i,j} u\Bigg\|_{1,k}\le \Lambda^2 \|u\|_{1,k},
\]
which, in turn, implies
\begin{align*}
    \sum_{i=1}^N \sum_{j=1}^{Q_i} \|P^f_{i,j} u\|_{1,k,\Omega^f_{i,j}}^2
	&= \sum_{i=1}^N \sum_{j=1}^{Q_i} (E^f_{i,j} P^f_{i,j} u,u)_{1,k} = \Biggl(\sum_{i=1}^N \sum_{j=1}^{Q_i} E^f_{i,j} P^f_{i,j} u,\,u\Biggr)_{1,k} \\
	&\le \Bigg\|\sum_{i=1}^N \sum_{j=1}^{Q_i} E^f_{i,j} P^f_{i,j} u\Bigg\|_{1,k}\|u\|_{1,k}
	\le \Lambda^2 \|u\|_{1,k}^2.
\end{align*}
This, together with \eqref{eq:10_1}, leads to
\begin{equation}\label{eq:10_2}
	-k^2\sum_{i=1}^N \sum_{j=1}^{Q_i} (E^f_{i,j} T^f_{i,j} u-u,E^f_{i,j} P^f_{i,j} u) \le 3kH^f \Lambda^3 \|u\|_{1,k}^2.
\end{equation}

\noindent
\textit{Step 4}.
We can now combine \eqref{eq: 4_9}, \eqref{eq:10_0} and \eqref{eq:10_2} to obtain
\begin{align*}
	\|u\|_{1,k}^2 &\le (2+3\Lambda^4\Theta)(Tu,u)_{1,k} 
	+ 2\Lambda^2(2+3\Lambda^4\Theta)\Bigl(2k\Theta^{\frac{1}{2}}(1+\Cstab)+3kH^f \Bigr)\|u\|_{1,k}^2
	\\[0.5ex]
	&= (2+3\Lambda^4\Theta)(Tu,u)_{1,k} - s\|u\|_{1,k}^2,
\end{align*}
which can be rearranged to prove \eqref{eq: 4_4}.

\medskip

\noindent
\textit{Step 5.}
In order to prove \eqref{eq: 4_5}, we start with the following inequality (since $E_0T_0 u = T_0u $),
\begin{equation}\label{eq: 4_18}
	\|Tu\|_{1,k}^2 =  \Bigg\|T_0 u + \sum_{i=1}^N \sum_{j=1}^{Q_i} E^f_{i,j} T^f_{i,j} u\Biggr\|_{1,k}^2 
	\le 2\|T_0 u\|_{1,k}^2 + 2\Bigg\|\sum_{i=1}^N \sum_{j=1}^{Q_i} E^f_{i,j} T^f_{i,j} u\Bigg\|_{1,k}^2 
\end{equation}
For the first term we use the Cauchy--Schwarz inequality and \eqref{3_7_d} and obtain
\begin{align*}
	\|T_0 u\|_{1,k}^2 &= (T_0 u, T_0 u)_{1,k} 
	= (T_0 u - u, T_0 u)_{1,k} + (u, T_0 u)_{1,k} 
	\\[0.5ex]
	&\le \|T_0 u - u\|_{1,k}\|T_0 u\|_{1,k} + \|u\|_{1,k}\|T_0 u\|_{1,k}
	\\[0.5ex]
	&\le 2\|u\|_{1,k}\|T_0 u\|_{1,k} + \|u\|_{1,k}\|T_0 u\|_{1,k} 
	= 3\|u\|_{1,k}\|T_0 u\|_{1,k}
\end{align*}
and hence
\begin{equation}\label{eq: 4_19}
	\|T_0 u\|_{1,k} \le 3 \|u\|_{1,k}. 
\end{equation}
For the second term in \eqref{eq: 4_18} we use Lemma~\ref{lemma_3_6}, which implies
\begin{equation}
	\label{eq:second}
	\Bigg\|\sum_{i=1}^N \sum_{j=1}^{Q_i} E^f_{i,j} T^f_{i,j} u\Bigg\|_{1,k}^2 
	\le \Lambda^2 \sum_{i=1}^N \sum_{j=1}^{Q_i} \|T^f_{i,j} u\|_{1,k,\Omega^f_{i,j}}^2 
	\le \Lambda^2 \sum_{i=1}^N \sum_{j=1}^{Q_i} 4\big\|u|_{\Omega_i}\big\|_{1,k,\Omega^f_{i,j}}^2 
	\le 4 \Lambda^3\|u\|_{1,k}^2. 
\end{equation}
Combining (\ref{eq: 4_19}) and \eqref{eq:second} with \eqref{eq: 4_18} we arrive at
\begin{equation*}
	\|Tu\|_{1,k}^2 \le 2\times9\|u\|_{1,k}^2 + 2\times4\Lambda^3\|u\|_{1,k}^2 
	= (18 + 8\Lambda^3)\|u\|_{1,k},
\end{equation*}
which proves \eqref{eq: 4_5}.

\end{proof}

It is now possible to continue with the proof for Theorem~\ref{theorem: convergence}.

\begin{proof}[Proof of Theorem~\ref{theorem: convergence}]
	Using \eqref{eq: 4_4} with \eqref{eq: 2_27} we obtain
	\begin{equation*}
		c_1 \|u\|_{1,k}^2 \le (Tu,u)_{k} = \bigl\langle \mathbf{M}_{AS,2}^{-1}\mathbf{B}\mathbf{u},\mathbf{u}\bigr\rangle_{\mathbf{D}_{k}},
	\end{equation*}
	which can be written as 
	\begin{equation*}
		c_1 \le \frac{\bigl\langle\mathbf{M}_{AS,2}^{-1}\mathbf{B}\mathbf{u},\mathbf{u}\bigr\rangle_{\mathbf{D}_{k}}}{\|\mathbf{u}\|_{\mathbf{D}_{k}}^2 }.
	\end{equation*}
	We can now use \eqref{eq: 4_5} again with \eqref{eq: 2_27} and get
	\begin{equation*}
		\|\mathbf{M}_{AS,2}^{-1}\mathbf{B}\mathbf{u}\|_{\mathbf{D}_{k}}^2 
		= \|Tu\|_{1,k}^2 
		\le c_2\|u\|_{1,k}^2 = c_2\|\mathbf{u}\|^2_{\mathbf{D}_{k}},
	\end{equation*}
	which implies
	\begin{equation*}
		\|\mathbf{M}_{AS,2}^{-1}\mathbf{B}\|_{\mathbf{D}_k}^2 \le c_2.
	\end{equation*}
	The result follow directly from Elman theory \cite{Elman:1983:VIM}
	with $\gamma=\frac{c_1}{c_2}$.
\end{proof}